\documentclass[12pt]{article}
\usepackage{amsmath, amsthm, amssymb, stmaryrd, fullpage}
\usepackage[all]{xy}

\newtheorem{theorem}{Theorem}[subsection]
\numberwithin{equation}{theorem}
\newtheorem{lemma}[theorem]{Lemma}
\newtheorem{cor}[theorem]{Corollary}
\newtheorem{prop}[theorem]{Proposition}
\theoremstyle{definition}
\newtheorem{defn}[theorem]{Definition}
\newtheorem{example}[theorem]{Example}
\newtheorem{convention}[theorem]{Convention}
\newtheorem{remark}[theorem]{Remark}
\newtheorem{hypothesis}[theorem]{Hypothesis}
\newtheorem{notation}[theorem]{Notation}
\newtheorem{question}[theorem]{Question}

\def\AAA{\mathbb{A}}

\def\FF{\mathbb{F}}

\def\PP{\mathbb{P}}
\def\QQ{\mathbb{Q}}
\def\RR{\mathbb{R}}
\def\ZZ{\mathbb{Z}}
\newcommand{\calE}{\mathcal{E}}
\newcommand{\calF}{\mathcal{F}}
\newcommand{\calG}{\mathcal{G}}
\newcommand{\calL}{\mathcal{L}}
\newcommand{\calO}{\mathcal{O}}

\newcommand{\gothm}{\mathfrak{m}}
\newcommand{\gotho}{\mathfrak{o}}

\newcommand{\Xbar}{\overline{X}}
\newcommand{\Ybar}{\overline{Y}}
\newcommand{\del}{\partial}

\def\dual{\vee}

\def\bv{\mathbf{v}}

\DeclareMathOperator{\alg}{alg}

\DeclareMathOperator{\Frac}{Frac}
\DeclareMathOperator{\Gal}{Gal}

\DeclareMathOperator{\GL}{GL}

\DeclareMathOperator{\Ind}{Ind}

\DeclareMathOperator{\Maxspec}{Maxspec}

\DeclareMathOperator{\rank}{rank}

\DeclareMathOperator{\Res}{Res}
\DeclareMathOperator{\sep}{sep}
\DeclareMathOperator{\Spec}{Spec}
\DeclareMathOperator{\Spf}{Spf}
\DeclareMathOperator{\Swan}{Swan}

\begin{document}

\title{Swan conductors for $p$-adic differential modules, II: 
Global variation}
\author{Kiran S. Kedlaya \\ Department of Mathematics, Room 2-165 \\ 
Massachusetts
Institute of Technology \\ 77 Massachusetts Avenue \\
Cambridge, MA 02139 \\
\texttt{kedlaya@mit.edu}}
\date{November 24, 2008}

\maketitle

\begin{abstract}
Using a local construction from a previous paper, we
exhibit a numerical invariant, the differential Swan conductor, 
for an 
isocrystal on a variety over a perfect field of positive characteristic
overconvergent
along a boundary divisor; this leads to an analogous construction for
certain $p$-adic and $\ell$-adic
representations of the \'etale fundamental group of 
a variety. We then demonstrate some variational properties
of this definition for overconvergent isocrystals, paying
special attention to the case of surfaces.
\end{abstract}

\section*{Introduction}

This paper is a sequel to \cite{kedlaya-swan1}, which defines a numerical
invariant, called the differential Swan conductor, for certain
differential modules on a rigid  analytic annulus over a $p$-adic field.
In that paper, the key application of the construction is the definition
of a sensible numerical invariant for
 Galois representations with finite local monodromy over an complete
discretely valued field of equal characteristic, without any assumption
of perfectness of the residue field.

In this paper, we adopt a more geometric viewpoint, taking the construction
back to its roots in the theory of $p$-adic cohomology.
We define differential
Swan conductors for an overconvergent isocrystal
on a variety over a perfect field of positive characteristic.
The definition depends on the choice of a boundary divisor along which one
measures the conductor;
we are particularly interested in understanding how the conductor can vary
as a function of this boundary divisor.
We give special attention to the
case of surfaces; one of the variational properties 
loosely resembles subharmonicity for functions on Berkovich analytic curves,
in the sense of Thuillier \cite{thuillier}.
Another resembles a semicontinuity property of \'etale cohomology 
\cite{laumon}. 

The variational properties of differential Swan conductors seem analogous
to properties of the irregularity
of a holomorphic differential module on a complex surface along a 
boundary divisor; indeed, a complex analogue of the semicontinuity
property mentioned above has recently been proved by Andr\'e
\cite{andre}, extending an old result of Deligne (see \cite{mebkhout}).
Variation of the boundary divisor has been studied in that
setting by
Sabbah \cite{sabbah}; our study was motivated by questions
in the $p$-adic realm
analogous to Sabbah's questions about the Stokes phenomenon. These arise
in the study of semistable reduction for overconvergent $F$-isocrystals,
which is the subject of an ongoing series of papers by the author
\cite{kedlaya-part1, kedlaya-part2, kedlaya-part3, kedlaya-part4}; 
in fact, some of the constructions 
used
in \cite{kedlaya-swan1} and in this paper already appear 
in \cite{kedlaya-part3}.
We have begun carrying these techniques over to Sabbah's setting
\cite{kedlaya-goodformal1}.

As in \cite{kedlaya-swan1}, there is
a mechanism for converting certain $p$-adic representations of the
\'etale fundamental group of a smooth variety into
$F$-isocrystals.
This makes it possible to define differential Swan conductors,
and (with some effort) to 
prove some of the corresponding properties, also for
lisse $\ell$-adic \'etale sheaves.

We end this introduction by cautioning that
this paper is not intended to be read
independently from \cite{kedlaya-swan1}. In particular, we freely use
notation and terminology introduced in \cite{kedlaya-swan1}, without
explicit reintroduction except in a few places for emphasis.

\subsection*{Acknowledgments}
Thanks to Yves Andr\'e, Matt Baker and Mattias Jonsson (independently), 
and Kazuhiro Fujiwara
for directing us to the work of Sabbah, Thuillier, and Kato,
respectively.
Thanks also to Liang Xiao for helpful feedback.
Some of this material was first presented at the Hodge Theory conference at
Venice International University in June 2006; that presentation 
was sponsored by the Clay Mathematics Institute.
The author was additionally supported by NSF grant DMS-0400727,
NSF CAREER grant DMS-0545904, a Sloan Research Fellowship,
and the NEC Research Support Fund.

\section{Relative annuli}
\label{sec:relative}

In this section, we gather some facts about the rigid geometry of
relative annuli (products of annuli with other spaces), 
in the vein of \cite[\S 3]{kedlaya-part1}.

\setcounter{theorem}{0}
\begin{hypothesis} \label{H:rational type}
Throughout this paper,
let $K$ be a complete nonarchimedean field of characteristic $0$
equipped with $m$ commuting continuous derivations 
$\del_1,\dots,\del_m$. Assume that $K$ is of \emph{rational type}
in the sense of \cite[Definition~1.5.3]{kedlaya-swan1},
i.e., there exist elements $u_1,\dots,u_m \in K$ such that:
\begin{itemize}
\item
for $i,j \in \{1,\dots,m\}$ with $i \neq j$, $\del_i(u_i) = 1$
and $\del_i(u_j) = 0$;
\item
for $n$ a positive integer, $i \in \{1,\dots,m\}$,
and $x \in F$, $|u_i^n \del_i^n(x)/n!| \leq |x|$.
\end{itemize}
Let $K_0$ be the joint kernel of $\del_1,\dots,\del_n$; it is again
a complete nonarchimedean field. Let $k_0, k$ be the respective 
residue fields of
$K_0, K$, and assume that $k_0, k$ are of characteristic $p>0$.
Let $\gotho_K$ denote the valuation subring of $K$,
let $\gothm_K$ denote the maximal ideal of $\gotho_K$,
and let $\Gamma^*$ be the divisible closure of $|K^\times|$.
\end{hypothesis}

\begin{remark}
Hypothesis~\ref{H:rational type} differs from
the running hypothesis \cite[Hypothesis~2.1.3]{kedlaya-swan1}
from the previous paper, which required one of the following.
(Beware that as written, \cite[Hypothesis~2.1.3(b)]{kedlaya-swan1}
is missing the hypothesis that $k$ is separable over $k_0$.)
\begin{enumerate}
\item[(a)]
$K$ is a finite unramified extension of the completion of
$K_0(u_1,\dots,u_n)$ for the $(1,\dots,1)$-Gauss norm.
\item[(b)]
$K_0$ and $K$ are discretely valued with the same value group,
$k$ is separable over $k_0$,
and $k$ admits a finite $p$-basis over $k_0$.
\end{enumerate}
On one hand, by \cite[Remark~1.5.10]{kedlaya-xiao}, both of these are special
cases of Hypothesis~\ref{H:rational type}. On the other hand,
all results in \cite{kedlaya-swan1} proved assuming
\cite[Hypothesis~2.1.3]{kedlaya-swan1} remain true assuming
Hypothesis~\ref{H:rational type}, with no change in 
the proofs (except that our
$K,K_0$ correspond to the labels $L,K$ in \cite{kedlaya-swan1}).
We will use results from \cite{kedlaya-swan1} generalized in this way
without further comment; see \cite[Theorem~2.6.1]{kedlaya-xiao} for a 
representative example of how the proofs carry over.
\end{remark}

\begin{hypothesis} \label{H:running}
Throughout this section:
\begin{itemize}
\item
let $P$ denote a smooth affine irreducible formal scheme
over $\Spf \gotho_K$, with generic fibre $P_K$ and special
fibre $Z = P_k$;
\item
let $L$ denote the completion of $\Frac \Gamma(P, \calO)$ for the topology
induced by the supremum norm on $P_K$;
\item
let $U$ denote an open dense subscheme of $Z$.
\end{itemize}
\end{hypothesis}

\begin{notation}
For $Z' \hookrightarrow Z$ an immersion, 
we denote by $]Z'[_P$ the inverse image
of $Z'$ under the specialization map $P_K \to Z$; we also refer to
$]Z'[_P$ as the \emph{tube}
of $Z'$ in $P_K$.
\end{notation}

\begin{defn}
We say a subinterval of $[0, +\infty)$ is \emph{aligned} if 
each endpoint at which it is closed belongs to
$\Gamma^* \cup \{0\}$. This is consistent with \cite[Notation~2.4.1]{kedlaya-part1},
which only applied to intervals not containing 0, and with
\cite[Definition~3.1.1]{kedlaya-part1}.
\end{defn}

\subsection{Relative annuli}

\begin{lemma} \label{L:rel annuli}
Let $Y$ be a rigid subspace of $P_K \times A_K[0,1)$. Then the following conditions are 
equivalent.
\begin{enumerate}
\item[(a)]
There exists $\epsilon \in (0,1)$ such that $P_K \times A_K(\epsilon,1)
\subseteq Y$.
\item[(b)]
There exists an affinoid subspace $V$ of $P_K \times A_K[0,1)$ such that
$\{Y,V\}$ forms an admissible covering of $P_K \times A_K[0,1)$.
\item[(c)]
There exist $\rho \in (0,1) \cap \Gamma^*$ and an
affinoid subspace $V$ of $P_K \times A_K[\rho,1)$ such that
$\{Y \cap (P_K \times A_K[\rho,1)),V\}$ forms an admissible covering of $P_K \times A_K[\rho,1)$.
\end{enumerate}
\end{lemma}
\begin{proof}
The implication (a)$\implies$(b) is clear: take $V = P_K \times A_K[0,\rho]$
for any $\rho \in (\epsilon,1) \cap \Gamma^*$.
The implication (b)$\implies$(c) is trivial.
For (c)$\implies$(a), note that the maximum modulus principle implies that
$t$ achieves its supremum $\eta$ on $V$, so $\eta$ must be less than 1;
we can thus satisfy (a) by choosing any $\epsilon \in (\eta,1)$.
\end{proof}

\begin{defn}
Define a \emph{relative annulus} over $P_K$ to be a subspace of 
$P_K \times A_K[0,1)$ satisfying one of the equivalent conditions of
Lemma~\ref{L:rel annuli}.
\end{defn}

\begin{defn} \label{D:generic fibre}
Given a coherent (locally free) sheaf $\calE$ on
a relative annulus $X$ containing
$P_K \times A_K(\epsilon,1)$,
there is a unique coherent (locally free) sheaf
$\calF$ on $A_L(\epsilon,1)$ such that for each closed aligned subinterval
$I$ of $(\epsilon,1)$, we have an identification
\[
\Gamma(A_L(I),\calF) \cong \Gamma(P_K \times A_K(I),\calE) \otimes_{\Gamma(P_K
\times A_K(I),\calO)} \Gamma(A_L(I),\calO),
\]
and these identifications commute with restriction maps.
We call $\calF$ the \emph{generic fibre} of $\calE$.
(See \cite[Definition~5.3.3]{kedlaya-part3} for more details.)
\end{defn}

The following lemma will be useful in consideration of generic fibres.
\begin{lemma} \label{L:same norm}
For $\rho \in (0,1) \cap \Gamma^*$,
suppose that $f \in \Frac \Gamma(P_K \times A_K[\rho,\rho],
\calO)$ can be written as a ratio of two elements of 
$\Gamma(P_K \times A_K[\rho,\rho], \calO)$, neither of which has a zero
in $A_L[\rho,\rho]$. Then the open dense subscheme $U$ of $Z$ can be
chosen so that for any $x \in ]U[_P \times A_K[\rho,\rho]$,
$|f(x)| = |f|_\rho$.
\end{lemma}
\begin{proof}
It  suffices to consider $f \in \Gamma(P_K \times A_K[\rho,\rho], 
\calO)$ having no zero in $A_L[\rho,\rho]$,
since by hypothesis
we can write the original $f$ as a quotient of two such functions.
Since $f$ has no zero in $A_L[\rho,\rho]$, its Newton polygon 
(in the sense of Lazard \cite{lazard}) has no
segment of the corresponding slope; that is, if we write $f = \sum_{i \in \ZZ}
c_i t^i$ with $c_i \in 
\Gamma(P_K, \calO)$, then there is a unique index $i$ with
$|c_i| \rho^i = |f|_\rho$. It thus suffices to check
the given assertion for $f = c_i t^i$, for which it is evident:
choose a scalar $\lambda \in K^\times$ such that $\lambda c_i$ belongs to
$\Gamma(P, \calO)$ and has nonzero image in $\Gamma(Z, \calO)$,
then take $U$ not meeting the zero locus of said image.
\end{proof}

\begin{defn} \label{D:extend Swan}
Let $X$ be a relative annulus over $P_K$ containing
$P_K \times A_K(\epsilon,1)$, let $\calE$ be a $\nabla$-module
on $X$ relative to $K_0$, 
and let $\calF$ be the generic fibre of $\calE$. Then
$\calF$ naturally admits the structure of a $\nabla$-module
on $A_L(\epsilon,1)$ relative to $K_0$, in the sense of
\cite[Definition~2.4.5]{kedlaya-swan1}.
We say that $\calE$ is \emph{solvable at $1$} if $\calF$ is,
in which case we define the \emph{highest break},
\emph{break multiset}, and
\emph{differential Swan conductor} of $\calE$ as the corresponding
items associated to $\calF$.
\end{defn}

\begin{remark} \label{R:extend Swan}
By the results of \cite[\S 2.6]{kedlaya-swan1},
the constructions in Definition~\ref{D:extend Swan} are invariant
under pullback along an automorphism of $P_K \times A_K[\epsilon,1)$
for $\epsilon \in (0,1) \cap \Gamma^*$, even if the automorphism does not
preserve $P_K$ or the projection onto $P_K$.
\end{remark}

\subsection{Fringed relative annuli}

We will have use for a variant of the concept of a relative annulus;
the resulting objects
are related to relative annuli in the same way that the weak formal
schemes of Meredith \cite{meredith}
are related to ordinary formal schemes,
or the dagger spaces of Grosse-Kl\"onne \cite{grosse-klonne}
are related to ordinary rigid spaces.

\begin{defn}
A \emph{strict neighborhood} of $]U[_P$ in $P_K$ is a rigid subspace
$W \subseteq P_K$ such that $\{W,  ]Z\setminus U[_P\}$ is an admissible
covering of $P_K$.
\end{defn}

\begin{lemma} \label{L:strict to generic}
Suppose that $Z \setminus U$ has pure codimension $1$ in $Z$.
Let $W$ be a strict neighborhood of $]U[_P$ in $P_K$.
Then for each $c \in (0,1)$, there exists a strict 
neighborhood $W' \subseteq W$
of $]U[_P$ in $P_K$ such that for each $f \in \Gamma(W, \calO)$,
\[
|f|_{\sup,W'} \leq |f|_{\sup,W}^c |f|_L^{1-c}.
\]
\end{lemma}
\begin{proof}
See \cite[Proposition~3.5.2]{kedlaya-part1}.
\end{proof}

\begin{defn} \label{D:fringed}
We say a subspace $Y \subseteq P_K \times A_K[0,1)$ is a
\emph{fringed relative annulus} over $]U[_P$ (or better,
over the inclusion $]U[_P \hookrightarrow 
P_K$) if 
$Y$ satisfies the
following property for some $\epsilon \in (0,1)$: for every 
closed aligned
subinterval
$I$ of $(\epsilon,1)$, there is a strict neighborhood $W$ of $]U[_P$
in $P_K$ 
such that $W \times A_K(I) \subset Y$.
\end{defn}

\begin{remark} \label{R:proper fringed}
We can use the same definition to define a fringed relative annulus
over $]U[_P$ for $P$ smooth \emph{proper} over $\Spf \gotho_K$.
We will have occasion to do this in Subsection~\ref{subsec:variation}.
\end{remark}

\begin{defn}
Let $Y$ be a fringed relative annulus over $]U[_P$ in $P_K$.
Then the intersection $Y_0 = Y \cap (]U[_P \times A_K[0,1))$ is 
a relative annulus over $]U[_P$; we call $Y_0$ the \emph{core} of $Y$.
We will extend various properties of relative annuli, or 
sheaves on relative
annuli, to fringed relative annuli by restriction to the core.
\end{defn}

\begin{lemma} \label{L:intersection2}
Suppose that $Z \setminus U$ has pure codimension $1$ in $Z$.
Let $W$ be a strict neighborhood of $]U[_P$.
Let $I$ be a closed aligned interval, let $I'$ be a closed
aligned subinterval of the interior
of $I$, and choose $\rho \in I' \cap \Gamma^*$. Then there exists
a strict neighborhood $W'$ of $]U[_P$ in $P_K$ such that
within $\Gamma(A_L[\rho,\rho], \calO)$,
\[
\Gamma(W \times A_K[\rho,\rho], \calO)
\cap \Gamma(A_L(I),\calO) \subseteq \Gamma(W' \times A_K(I'), \calO).
\]
\end{lemma}
\begin{proof}
Write $I = [a,b]$ and $I' = [a',b']$, so that $a < a' \leq \rho \leq b' < b$. 
Note that for $c \in (0,1)$
sufficiently close to 0, we have
\begin{equation} \label{eq:rho choice}
\rho^c a^{1-c} < a', \qquad b' < \rho^c b^{1-c}.
\end{equation}
Fix one such $c$; by
Lemma~\ref{L:strict to generic}, we can choose the strict
neighborhood $W'$ so that for any $f \in \Gamma(W,\calO)$,
\begin{equation} \label{eq:convexity}
|f|_{\sup,W'} \leq |f|_{\sup,W}^c |f|_L^{1-c}.
\end{equation}
For $f \in \Gamma(A_L[\rho,\rho], \calO)$, we can write
$f = \sum_{i \in \ZZ} f_i t^i$ with $f_i \in L$. If
$f \in \Gamma(W \times A_K[\rho,\rho],\calO)$, then
$f_i \in \Gamma(W, \calO)$ for each $i$, and
$|f_i|_{\sup,W} \rho^i \to 0$ as $i \to \pm \infty$;
if $f \in \Gamma(A_L(I), \calO)$, then for each $\eta \in I$,
$|f_i|_L \eta^i \to 0$ as $i \to \pm \infty$.
If both containments hold, then by \eqref{eq:convexity},
\[
\lim_{i \to \pm \infty} |f_i|_{\sup,W'} (\rho^c \eta^{1-c})^i = 0 \qquad (\eta \in
I);
\]
by \eqref{eq:rho choice}, this implies that for any $\eta \in I'$,
$|f_i|_{\sup,W'} \eta^i \to 0$ as $i \to \pm \infty$. This proves the claim.
\end{proof}

\begin{lemma} \label{L:proj intersect}
Let 
\[
\xymatrix{
R \ar[r] \ar[d] &  S \ar[d] \\
T \ar[r] & U
}
\]
be a commuting diagram of inclusions of integral domains,
such that the intersection
$S \cap T$ within $U$ is equal to $R$. Let $M$ be a finite
locally free $R$-module. Then the intersection of $M \otimes_R S$
and $M \otimes_R T$ within $M \otimes_R U$ is equal to $M$.
\end{lemma}
\begin{proof}
See \cite[Lemma~2.3.1]{kedlaya-xiao}.
%
\end{proof}

\begin{lemma} \label{L:intersection3}
Suppose that $Z \setminus U$ has pure codimension $1$ in $Z$.
Let $Y$ be a fringed relative annulus over $]U[_K$,
choose $\epsilon \in (0,1)$ as in Definition~\ref{D:fringed},
choose $\rho \in (\epsilon,1) \cap \Gamma^*$,
and choose a strict neighborhood $W$ of $]U[_K$ in $P_K$ such that
$W \times A_K[\rho,\rho] \subset Y$.
Let $\calE$ be a coherent locally free sheaf on $Y$, and let
$\calF$ be the generic fibre of $\calE$ on $A_L(\epsilon,1)$.
Suppose that $\bv \in \Gamma(A_L[\rho,\rho], \calF)$ satisfies
\[
\bv \in \Gamma(W \times A_K[\rho,\rho],\calE)
\cap \Gamma(A_L(\epsilon,1), \calF).
\]
Then there exists a fringed relative annulus $Y'$ over $]U[_K$
such that $\bv \in \Gamma(Y', \calE)$.
\end{lemma}
\begin{proof}
It suffices to show that for each closed 
aligned subinterval $I'$ of $(\epsilon,1)$
containing $\rho$,
there exists a strict neighborhood $W'$ of $]U[_K$ in $P_K$ such that
$\bv \in \Gamma(W' \times A_K(I'), \calE)$.
Choose a closed aligned interval $I$ of $(\epsilon,1)$ containing $I'$ in its interior;
by Lemma~\ref{L:intersection2}, we can choose $W'$ so that
within $\Gamma(A_L[\rho,\rho], \calO)$,
\[
\Gamma(W \times A_K[\rho,\rho], \calO)
\cap \Gamma(A_L(I),\calO) \subseteq \Gamma(W' \times A_K(I'), \calO).
\]
We may then apply Lemma~\ref{L:proj intersect} to deduce the claim.
\end{proof}

\subsection{Globalizing the break decomposition}

The main result of this subsection
(Theorem~\ref{T:break decomp2})
is a globalized version of \cite[Theorem~2.7.2]{kedlaya-swan1}.
To prove it, we use 
the following relative version of \cite[Lemma~2.7.10]{kedlaya-swan1}.
\begin{prop} \label{P:hard case}
Suppose $u_{m+1}, \dots, u_n \in \Gamma(P, \calO)$ are such that
$du_{m+1}, \dots, du_n$ freely generate $\Omega^1_{P/\gotho_K}$ over 
$\gotho_K$. Write
$\del_1, \dots, \del_{n+1}$
for the basis of derivations on $P_K \times A_K[0,1)$ over $K_0$
dual to $du_1, \dots, du_n, dt$. 
Let $Y$ be a fringed relative annulus over $]U[_P$.
Let $\calE$ be a $\nabla$-module on $Y$ which is solvable at $1$.
Let $\calF$ denote the generic fibre of $\calE$, viewed as a 
$\nabla$-module relative to $K_0$, and choose
$i \in \{1,\dots,n+1\}$ such that $\del_i$ is eventually dominant for
$\calF$. Suppose that there exist $\rho \in (0,1)$ arbitrarily close to 
$1$ such that the scale multiset for $\del_i$ on $\calF_\rho$ contains 
more than one element. Then after shrinking $U$ 
(to another open dense subscheme of $Z$) and $Y$
(to a fringed relative annulus over $]U[_P$),
$\calE$ becomes decomposable.
\end{prop}
\begin{proof}
We first treat the case $i=n+1$.
Let $b$ be the highest break of $\calF$.
By \cite[Theorem~2.7.2 and Remark~2.7.7]{kedlaya-swan1},
we may choose $\epsilon \in (0,1)$ such that 
$\calF$ admits a break decomposition over $A_L(\epsilon,1)$,
and for all $\rho \in (\epsilon,1)$, $\del_{n+1}$ is dominant for $\calF_\rho$ 
and $T(\calF,\rho) = \rho^b$.
Pick a closed aligned interval $I \subset (\epsilon,1)$ of
positive length for which there exists a nonnegative integer $m$
such that 
\[
|p|^{p^{-m+1}/(p-1)} < T(\calF, \rho) < |p|^{p^{-m}/(p-1)} \qquad
(\rho \in I).
\]
Let $\calF_m$ be the $\nabla$-module on $A_L(I^{p^m})$ which is the
$m$-fold Frobenius antecedent of $\calF$
in the $t$-direction, as produced by \cite[Theorem~6.15]{kedlaya-mono-over},
so that 
\[
T(\calF_m,\rho^{p^m}) = T(\calF,\rho)^{p^m} < |p|^{1/(p-1)}.
\]
Since the defining inequality for Frobenius antecedents is strict,
Lemma~\ref{L:strict to generic} allows us to correspondingly construct
an $m$-fold Frobenius antecedent $\calE_m$ of $\calE$ on
$W \times A_K(I^{p^m})$ for some strict neighborhood $W$
of $]U[_P$ in $P_K$.

Choose a cyclic vector for $\calE_m$ with respect to
$\del_{n+1}$ over $\Frac \Gamma(W \times A_K(I^{p^m}), \calO)$,
and let $Q = T^d + \sum_{i=0}^{d-1}
a_i T^i$ be the corresponding twisted polynomial.
Pick $\rho \in I \cap \Gamma^*$ such that 
each $a_i$ can be written as a ratio of two elements of 
$\Gamma(W \times A_K(I^{p^m}), \calO)$, neither having any zeroes
in $A_L[\rho^{p^m}, \rho^{p^m}]$; 
the restriction excludes only finitely
many $\rho$.
By Lemma~\ref{L:same norm}, after shrinking $U$,
each of the $a_i$ becomes invertible on $]U[_P \times 
A_K[\rho^{p^m},\rho^{p^m}]$, and 
the norm of $a_i(x)$ for each  $x \in ]U[_P \times A_K[\rho^{p^m},
\rho^{p^m}]$ equals the supremum norm of $a_i$ in
$A_L[\rho^{p^m}, \rho^{p^m}]$.

Let $a_{j}$ be the coefficient at which the Newton polygon of $Q$
with respect to the supremum norm on $]U[_P \times A_K[\rho^{p^m}, \rho^{p^m}]$
has its first breakpoint (i.e., the one separating the segment of least
slope).
By a suitable application of Lemma~\ref{L:strict to generic}, 
we see that after shrinking $W$,
$a_j$ is also a breakpoint (though maybe not the first)
when computing slopes of $Q$
using the supremum norm on
$W \times A_K[\rho^{p^m},\rho^{p^m}]$.
Using Christol's factorization theorem
\cite[Theorem~2.2.2]{kedlaya-course}
for the supremum norm on $W \times A_K[\rho^{p^m},\rho^{p^m}]$
(and otherwise arguing as in \cite[Theorem~6.4.4]{kedlaya-course}),
we deduce that the factorization
of $Q$ provided by \cite[Proposition~1.1.10]{kedlaya-swan1}
that splits off the least $d-j$ slopes (counting multiplicity)
is defined over $W \times A_K[\rho^{p^m},\rho^{p^m}]$.
By performing the same argument again in the opposite twisted polynomial
ring (as in the proof of \cite[Proposition~3.3.10]{kedlaya-part3}),
we obtain a projector in $\calE_m^\dual \otimes \calE_m$
on $W \times A_K[\rho^{p^m},\rho^{p^m}]$.
This pulls back to a projector in $\calE^\dual \otimes \calE$ on
$W \times A_K[\rho,\rho]$; since the projector is already defined on
$A_L(\epsilon,1)$, by Lemma~\ref{L:intersection3} it becomes defined on $Y$
after shrinking $Y$. This proves the desired decomposability.

We now suppose $i \neq n+1$. We perform rotation as 
in the proof of \cite[Lemma~2.7.10]{kedlaya-swan1};
that is, first pull back along a map effecting $t \mapsto t^{p^N}$
for $N$ a suitably large integer,
then along a map effecting $u_i \mapsto u_i + t$. (Note that both of
these extend to maps between suitable fringed relative annuli:
by Lemma~\ref{L:strict to generic}, the series in
\cite[Definition~2.6.2]{kedlaya-swan1}
converges on some fringed relative annulus.)
As in the proof of \cite[Lemma~2.7.10]{kedlaya-swan1},
the decomposition obtained after rotation descends back to $\calE$.
\end{proof}

\begin{theorem} \label{T:break decomp2}
Let $Y$ be a fringed relative annulus over $]U[_P$.
Let $\calE$ be a $\nabla$-module on $Y$ which is solvable at $1$.
Then after shrinking $U$
(to another open dense subscheme of $Z$) and $Y$
(to a fringed relative annulus over $]U[_P$),
there exists a unique decomposition 
\[
\calE = \bigoplus_{b \in \QQ_{\geq 0}} \calE_b
\]
of $\nabla$-modules on $Y$, such that $\calE_b$ has uniform break $b$.
\end{theorem}
\begin{proof}
The claim is local on $Z$, so we may reduce to the case where
there exist $u_{m+1},\dots,u_n \in \Gamma(P,\calO)$ such that
$du_{m+1},\dots,du_n$ freely generate $\Omega^1_{P/\gotho_K}$ over $\gotho_K$.
After shrinking $U$ and $Y$,
we may reduce to the case where $\calE$ remains indecomposable after
further shrinking of $U$ and $Y$.
In this case, $\calE$ is forced to have a uniform break by
Proposition~\ref{P:hard case}.
\end{proof}

\begin{defn}
For $z \in Z$, we say that \emph{the break decomposition of $\calE$
extends across $z$} if we can choose $U$ in Theorem~\ref{T:break decomp2}
to contain $z$.
\end{defn}

We will need a criterion for detecting when the break decomposition
extends across $z$.
\begin{lemma} \label{L:extend break}
With notation as in Theorem~\ref{T:break decomp2},
let $K_\rho, L_\rho$ be the completions of $K(t), L(t)$ 
for the $\rho$-Gauss norm.
Suppose that $z \in U$ and that for each $\rho \in (\epsilon,1)$ sufficiently
close to $1$, the restriction of $\calE$ to $]U[_K \times_K K_\rho$
admits a decomposition whose restriction to $L_\rho$ coincides with
the restriction of the break decomposition of the generic fibre of
$\calE$. Then the break decomposition of $\calE$ extends
across $z$.
\end{lemma}
\begin{proof}
For each closed interval $I \subset (\epsilon,1)$ containing $\rho$,
inside $L_\rho$ we have
\[
\Gamma(]U[_K \times_K K_\rho, \calO) \cap
\Gamma(A_L(I), \calO) = \Gamma(]U[_K \times A_K(I), \calO).
\]
We may thus deduce the claim from Lemma~\ref{L:proj intersect}.
\end{proof}

\section{Representations, isocrystals, and conductors}

In this section, we define the differential highest break and
Swan conductor
associated to an isocrystal on a $k$-variety $X$ and
a boundary divisor in some compactification of $X$ along which the isocrystal
is overconvergent. We then show how a
special class of overconvergent isocrystals, those admitting
unit-root Frobenius actions, relate closely
to representations of the \'etale fundamental group of $X$.
This allows us to define differential ramification breaks and Swan conductors
for an appropriate class of $p$-adic representations, including
discrete representations (those with open kernel).

\setcounter{theorem}{0}
\begin{convention}
For the rest of this 
paper, a \emph{variety over $k$} will be a reduced separated
(but not necessarily irreducible) scheme of finite type over $k$,
and \emph{points} of a variety will always be closed points 
unless otherwise specified.
\end{convention}

\subsection{Convergent and overconvergent isocrystals}

This is not the place to reintroduce the full theory of convergent and
overconvergent isocrystals; we give here merely a quick summary.
See \cite{kedlaya-part1} for a less hurried
review, or \cite{berthelot} for a full development.

\begin{defn} \label{D:overconvergent}
Let $P$ be an affine formal scheme of finite type over $\Spf \gotho_K$ with
special fibre $Y$. Let $X$ be an open dense subscheme of $Y$ 
such that $Z = Y \setminus X$ is of pure codimension 1 in $Y$,
and $P$ is smooth over $\gotho_K$ in a neighborhood of $X$;
let $Q$ be the open formal subscheme of $P$ with special fibre $X$.
An \emph{isocrystal on $X$ overconvergent along $Z$}
is a $\nabla$-module $\calE$ relative to $K_0$
on a strict neighborhood of $]X[_P$ in $P_K$,
whose formal Taylor isomorphism converges on a strict neighborhood of
$]X[_{P \times P}$ in $P_K \times P_K$;
morphisms between these should likewise be defined on some strict 
neighborhood.
This definition turns out to be canonically independent of the choices of
$P$, so extends to arbitrary pairs $(X, Y)$ where $X$ is an
open dense subscheme of $Y$ smooth over $k$,
and $Y \setminus X$ is of pure codimension 1 in $Y$.
(The codimension 1 condition can be eased with a bit more work.)
If $Y$ is proper, then the category of isocrystals on  $X$
overconvergent along $Y \setminus X$ is independent of the choice of $Y$;
we call such objects \emph{overconvergent isocrystals} on $X$.
If on the other hand $Y = X$, we say $\calE$ is a \emph{convergent isocrystal}
on $X$.
\end{defn}

\begin{remark} \label{R:relative}
The usual definition of an isocrystal involves a $\nabla$-module relative to
$K$, not $K_0$. In fact, there is no harm in adding this
extra data:
the Taylor isomorphism is determined by the connection relative to $K$,
so it is harmless to carry the extra components of the connection through
the arguments in \cite{berthelot}. 
The construction relative to a subfield is useful 
for certain arguments where one wants to reduce the dimension
of a variety without losing critical data about the connection.
See Theorem~\ref{T:polyannuli convex2} for an argument of this form.
\end{remark}

\begin{defn} \label{D:Frobenius struct}
Let $\phi_K$ be a $q$-power Frobenius lift on $K$ acting on $K_0$;
that is, $\phi_K$ is an isometric endomorphism of $K$ acting on $K_0$,
and its action on $k$ is the $q$-power absolute Frobenius. 
With notation as in Definition~\ref{D:overconvergent}, a 
\emph{Frobenius structure}
on an isocrystal $\calE$ on $X$ overconvergent along $Z$
is an isomorphism $F: \phi^* \calE
\cong \calE$, for $\phi$ a $\phi_K$-semilinear $q$-power
Frobenius lift on $Q$;
note that $\phi$ extends to a strict neighborhood of $Q_K$ in $P_K$, so that
it makes sense to require $F$ to be an isomorphism of overconvergent 
isocrystals.
The word \emph{$F$-isocrystal} is shorthand for \emph{isocrystal with
Frobenius structure}.
\end{defn}

\begin{prop} \label{P:full faith}
Let $X \hookrightarrow Y$ be an open immersion of $k$-varieties with dense
image, with $X$ smooth and $Y \setminus X$ of pure codimension $1$ in $Y$. Then
the restriction functor from the category of $F$-isocrystals
on $X$ overconvergent along $Y \setminus X$ to the category
of convergent $F$-isocrystals on $X$ is
fully faithful.
\end{prop}
\begin{proof}
It suffices to check relative to $K$, in which case this assertion becomes
\cite[Theorem~4.2.1]{kedlaya-part2}.
\end{proof}

\begin{defn}
With notation as in Definition~\ref{D:Frobenius struct}, we say that
$\calE$ is \emph{unit-root} if for each closed point $x \in X$, the
pullback of $\calE$ to $x$, which we may view as a finite dimensional 
$K$-vector
space $V_x$ equipped with a $\phi_K$-semilinear endomorphism $\phi$, admits
a $\gotho_K$-lattice $T$ such that $\phi$ induces an isomorphism $\phi_K^*(T)
\cong T$.
\end{defn}

\subsection{Globalizing the Swan conductor}
\label{subsec:globalization}

Much as the calculations on relative annuli in \cite[\S 3]{kedlaya-part1} were used
later therein
to define notions of constant/unipotent local monodromy for overconvergent
isocrystals, we can define differential Swan conductors for overconvergent
isocrystals as follows. 
(See \cite[\S 4]{kedlaya-part1} for a similar construction.)

\begin{defn} \label{D:geometric Swan}
Let $\overline{X}$ 
be a smooth $k$-variety, let $Z$ be a smooth irreducible
divisor on $\overline{X}$,
and let $\calE$ be an isocrystal on $X = \overline{X} \setminus Z$
overconvergent along $Z$. Suppose for the moment that 
there exists a smooth irreducible affine formal scheme $Q$ over $\Spf
\gotho_K$
with $Q_k \cong \overline{X}$; then $\calE$ can be 
realized as a $\nabla$-module
on some strict neighborhood $V$ of $]X[_Q$ in  $Q_K$,
as in Definition~\ref{D:overconvergent}.
Moreover, $]Z[_Q$ is a relative annulus by Berthelot's fibration theorem  
\cite[Th\'eor\`eme~1.3.7]{berthelot},
\cite[Proposition~2.2.9]{kedlaya-part1}, 
as then is $W = V \cap ]Z[_Q$ after appropriately shrinking $V$.
The overconvergence property forces 
the restriction of $\calE$ to $W$ to be solvable at 1, so $\calE$ admits
a break multiset and Swan conductor (relative to $K_0$)
via Definition~\ref{D:extend Swan}.

Now go back and note that the construction persists under restricting from $X$
to an open neighborhood of any given point of $Z$. 
Moreover, by Remark~\ref{R:extend Swan}, there is no dependence on how
$]Z[_Q$ is viewed as a relative annulus. (This implies independence from 
the choice of $Q$ itself, since $Q$ is unique up to noncanonical isomorphism
by \cite[Proposition~1.4.3]{arabia}.)
Consequently,
the definitions
extend unambiguously even if $\overline{X}$ is reducible or
does not lift globally. We write $b_i(\calE,Z)$ and $\Swan(\calE, Z)$ for the 
differential ramification
breaks (listed in decreasing order as $i$ increases) and
differential Swan conductor of $\calE$ along $Z$.
\end{defn}

\begin{remark} \label{R:diff Swan}
If $k$ is perfect,
$X$ is a smooth irreducible $k$-variety, 
$\calE$ is an overconvergent isocrystal on $X$,
and $v$ is any divisorial valuation on the function field $k(X)$ over $k$,
then we can also define the break multiset and Swan conductor of
$\calE$ along $v$, by blowing up into the case where $v$ is centered
on a generically 
smooth divisor, then applying Definition~\ref{D:geometric Swan}.
If $k$ is imperfect, then the previous discussion applies unless
blowing up gives a divisor which is geometrically nonreduced.
If $\calE$ is only overconvergent along the boundary of some
partial compactification $\overline{X}$ of $X$, then the previous
discussion applies to divisorial valuations which are centered on
$\overline{X}$. (That is, there must exist some blowup of $\overline{X}$
on which the valuation corresponds to the order of vanishing along an
irreducible divisor.)
\end{remark}

\subsection{\'Etale fundamental groups and unit-root isocrystals}

\begin{hypothesis} \label{H:etale1}
Throughout this subsection,
fix a power $q$ of $p$, and assume that the field $k = k_0$ 
is perfect and contains 
$\FF_q$. Assume also that $K = K_0$ is discretely valued, and comes
equipped with a $q$-power Frobenius lift $\phi_K$.
Let $K^\phi$ denote the fixed field of $K$ under $\phi$;
it is a complete discretely valued field with residue field $\FF_q$.
\end{hypothesis}

\begin{hypothesis}\label{H:etale2}
Throughout this subsection,
let $X$ be a smooth irreducible 
$k$-variety and let $\overline{x}$ be a geometric
point of $X$. We write $\pi_1(X, \overline{x})$ for the \'etale fundamental
group of $X$ with basepoint $\overline{x}$.
\end{hypothesis}

\begin{convention}
By a \emph{$p$-adic representation} of $\pi_1(X, \overline{x})$, we will 
mean a continuous homomorphism $\rho: \pi_1(X, \overline{x}) \to \GL(V)$
for $V = V(\rho)$ a finite dimensional $K^\phi$-vector space.
\end{convention}

The following result is due to Crew \cite[Theorem~2.1]{crew}.
\begin{theorem} \label{T:crew functor}
There is a natural equivalence of categories 
(functorial in $X$)
between the category of $p$-adic representations of $\pi_1(X, \overline{x})$
and the category of convergent unit-root $F$-isocrystals on $X$.
\end{theorem}
Crew also posed the question of identifying which $p$-adic representations
correspond to overconvergent unit-root $F$-isocrystals on $X$. For $X$ a curve,
this was answered by Tsuzuki \cite{tsuzuki-finite}; the hard work in the
general case is already present in Tsuzuki's work. 
All we need to add is a bit of analysis
of extendability for overconvergent isocrystals, from \cite{kedlaya-part1}.

\begin{defn}
Let $v$ be a divisorial valuation on the function field $k(X)$ over $k$,
and let $k(X)_v$ be the completion of $k(X)$ under $v$.
Fix a separable closure $k(X)_v^{\sep}$ of $k(X)_v$
and a perfect closure $k(X)_v^{\alg}$ of $k(X)_v^{\sep}$,
and let $\overline{x}$ be the geometric point of $X$
corresponding to the inclusion $k(X) \hookrightarrow k(X)_v^{\alg}$.
Put $\eta = \Spec k(X)_v$; then the morphism
$\eta \to X$ 
corresponding to the inclusion $k(X) \hookrightarrow 
k(X)_v$ induces a homomorphism
$\iota: \pi_1(\eta, \overline{x}) \to \pi_1(X, \overline{x})$,
and the former group may be canonically identified with
$\Gal(k(X)_v^{\sep}/k(X)_v)$.
Let $I_v$ be the inertia subgroup of 
$\Gal(k(X)_v^{\sep}/k(X)_v)$, i.e., the subgroup acting trivially
on the residue field of $k(X)_v^{\sep}$;
we refer to any subgroup of $\pi_1(X, \overline{x})$
conjugate to $\iota(I_v)$ as an \emph{inertia subgroup}
corresponding to $v$.
\end{defn}

\begin{defn}
We say a $p$-adic representation $\rho$ of $\pi_1(X, \overline{x})$ is
\emph{unramified} if every inertia subgroup of $\pi_1(X, \overline{x})$
lies in the kernel of $\rho$.
If $X$ admits a dense
 open immersion into a smooth proper irreducible $k$-variety
$\overline{X}$ (as would be ensured
by a suitably strong form of resolution of singularities in positive 
characteristic), then by Zariski-Nagata purity \cite[Expos\'e~X,
Th\'eor\`eme~3.1]{sga1},
$\rho$ is unramified if and only if $\rho$ factors through
$\pi_1(\overline{X}, \overline{x})$.
We say $\rho$ is 
\emph{potentially unramified} if there exists a finite \'etale cover
$Y$ of $X$ such that for any geometric point $\overline{y}$ of $Y$ 
over $\overline{x}$, the restriction of $\rho$ to $\pi_1(Y, \overline{y})$
is unramified (it suffices to check for a single $\overline{y}$).
\end{defn}

\begin{theorem} \label{T:crew functor2}
The functor of Theorem~\ref{T:crew functor} induces an equivalence between
the category of potentially unramified $p$-adic representations
of $\pi_1(X, \overline{x})$,
and the category of overconvergent unit-root $F$-isocrystals on $X$.
\end{theorem}
\begin{proof}
We first show that every representation $\rho$ corresponding to an 
overconvergent unit-root $F$-isocrystal is potentially unramified.
Choose a $\rho$-stable 
$\gotho_{K^\phi}$-lattice $T$ in $V = V(\rho)$; then there is a unique
finite \'etale Galois cover $Y$ of $X$ such that for any
geometric point $\overline{y}$ of $Y$ over $\overline{x}$,
$\pi_1(Y, \overline{y})$ equals the kernel of the action of $\rho$
on $T/2 p T$. By \cite[Proposition~7.2.1]{tsuzuki-duke}, the
intersection of $\pi_1(Y, \overline{y})$ with any inertia subgroup
of $\pi_1(X, \overline{x})$ belongs to the kernel of $\rho$; hence
$\rho$ is potentially unramified.

We next show that every potentially unramified $\rho$ corresponds to
an overconvergent unit-root $F$-isocrystal.
Let $\calE$ be the convergent unit-root $F$-isocrystal on $X$
corresponding to $\rho$.
Choose a finite \'etale Galois cover $f: Y \to X$ such that
for any geometric point $\overline{y}$ of $Y$ over $\overline{x}$,
the restriction of $\rho$ to $\pi_1(Y, \overline{y})$ is unramified.
By de Jong's alterations theorem \cite[Theorem~4.1]{dejong}, 
there exists an open dense subscheme $U$ of $X$ and a finite
\'etale cover $g: Z \to f^{-1}(U)$ such that $Z$ admits a dense open immersion
into a smooth proper $k$-variety $\overline{Z}$.
There is no harm in moving the basepoints $\overline{x}$ and $\overline{y}$
so that $\overline{x} \in U$; then for any geometric point $\overline{z}$
of $Z$ over $\overline{x}$,
the restriction of $\rho$ to $\pi_1(Z, \overline{z})$ is again unramified,
so factors through $\pi_1(\overline{Z}, \overline{z})$.

By Theorem~\ref{T:crew functor}, this restriction of $\rho$ corresponds to
a convergent unit-root
$F$-isocrystal $\calF$
on $\overline{Z}$. Since $\overline{Z}$ is proper,
there is no distinction between convergent and overconvergent 
on $\overline{Z}$, so we may restrict $\calF$ to an overconvergent
$F$-isocrystal on $Z$. Now put $\calG = f_* g_* \calF$, which is 
an overconvergent unit-root $F$-isocrystal on $U$ (see \cite[\S 5]{tsuzuki-duke}
for the pushforward construction). Let $\sigma$ be the $p$-adic
representation of $\pi_1(U, \overline{x})$
corresponding to $\calG$; then adjunction and trace give
$\pi_1$-equivariant maps $V(\rho) \to V(\sigma) \to V(\rho)$
whose composition is the identity. Composing the other way gives a
projector on $V(\sigma)$, corresponding to a projector on $\calG$ in the
category of convergent unit-root $F$-isocrystals on $U$.
By Proposition~\ref{P:full faith}, this projector actually exists in the
overconvergent category; its image is an overconvergent unit-root
$F$-isocrystal on $U$ which becomes isomorphic to
$\calE$ in the convergent category.
By \cite[Proposition~5.3.7]{kedlaya-part1}, that isomorphism ensures that
$\calE$ is the restriction to $U$ of an overconvergent unit-root $F$-isocrystal
on $X$, as desired.
\end{proof}

Theorem~\ref{T:crew functor2} can also be stated for partially overconvergent
isocrystals.

\begin{defn}
Let $X \hookrightarrow \overline{X}$ 
be an open immersion of $k$-varieties with dense
image, with $X$ smooth irreducible.
 We say a $p$-adic representation $\rho$ of $\pi_1(X, \overline{x})$ is
\emph{unramified on $\overline{X}$} 
if every inertia subgroup of $\pi_1(X, \overline{x})$
corresponding to a divisorial valuation centered on $\overline{X}$
lies in the kernel of $\rho$. We say $\rho$ is \emph{potentially unramified
on $\overline{X}$} 
if there exists a connected finite cover $f: \overline{Y} \to \overline{X}$
\'etale over $X$, such that for any geometric point $\overline{y}$ of
$Y = f^{-1}(X)$, the restriction of $\rho$ to $\pi_1(Y, \overline{y})$
is unramified on $\overline{Y}$.
\end{defn}

\begin{theorem}
The functor of Theorem~\ref{T:crew functor} induces an equivalence between
the category of $p$-adic representations of $\pi_1(X, \overline{x})$
potentially 
unramified on $\overline{X}$,
and the category of unit-root $F$-isocrystals on $X$ overconvergent
along $\overline{X} \setminus X$.
\end{theorem}
\begin{proof}
The proof is as in Theorem~\ref{T:crew functor2}. Note that 
the case $X = \overline{X}$ is Theorem~\ref{T:crew functor} itself,
while the case where $\overline{X}$
is proper over $k$ is Theorem~\ref{T:crew functor2}.
\end{proof}

\begin{remark}
One can also use the construction of Abbes and Saito 
\cite{abbes-saito1, abbes-saito2} to define Swan
conductors for $p$-adic representations. It has been shown recently
by Xiao \cite{xiao} that this construction agrees with 
the differential Swan conductor. Consequently, the results we obtain
about differential Swan conductors will apply also to Abbes-Saito
conductors. This agreement also occurs in the $\ell$-adic setting, as
discussed in Section~\ref{sec:l-adic}.
(In \cite{xiao2}, Xiao gives an analogue of differential Swan conductors
in mixed characteristic, and obtains an analogous comparison theorem
with Abbes-Saito conductors.)
\end{remark}

\subsection{Normalization of conductors}

When studying variation of differential Swan conductors, it will be useful
to normalize as follows.
\begin{defn} \label{D:normalization}
Let $X$ be a smooth irreducible $k$-variety, 
let $X \hookrightarrow \overline{X}$ be an open immersion of $k$-varieties
with dense image, and let $\calE$ be an
isocrystal on $X$
overconvergent along $\overline{X} \setminus X$.
As noted in Remark~\ref{R:diff Swan}, we can define
the differential ramification breaks and
the differential Swan conductor of $\calE$ with respect to a suitable divisorial 
valuation
$v$ on $k(X)$ over $k$ centered on $\overline{X}$
(which may be arbitrary if $k$ is perfect); we refer to
these as being in their \emph{natural normalization}.
For $t \in k(X)^*$ with $v(t) \neq 0$, 
we define the \emph{normalization with respect to $t$}
of the differential ramification breaks, or the differential Swan conductor, 
with respect to $v$ as the natural
normalization divided by the index of $v(t)\ZZ$ in the value group of $v$.
\end{defn}

For an easy example, we return to the Dwork isocrystals of
\cite[Example~3.5.10]{kedlaya-swan1}, but this time in a global setting.
\begin{defn} \label{D:dwork iso}
Assume that $K$ contains an element $\pi$ with $\pi^{p-1} = -p$
(a \emph{Dwork pi}).
Let $\calL$ be the $\nabla$-module of rank 1
on $\AAA^1_K$ with $\nabla$-action given on a generator $\bv$ by
\[
\nabla(\bv) = \pi \bv \otimes dt.
\]
One shows by a direct calculation that $\calL$ gives an overconvergent
$F$-isocrystal on $\AAA^1_k$, called the \emph{(standard) Dwork isocrystal};
it is in fact the image under the functor of Theorem~\ref{T:crew functor2}
of a nontrivial character of the Artin-Schreier cover
$\Spec k[z,t]/(z^p - z -t) \to \Spec k[t] = \AAA^1_k$.
For $X$ any variety over $k$ and $f \in \Gamma(X,\calO)$, we may identify
$f$ with a regular map $X \to \AAA^1_k$, and define
$\calL_f$ as the pullback $f^* \calL$, as an overconvergent $F$-isocrystal
on $X$.
\end{defn}

\begin{example} \label{ex:Dwork}
Assume $k=k_0$,
let $\calE$ be the Dwork isocrystal $\calL_{xy}$ on $\AAA^2_k$,
and compute conductors using \cite[Example~3.5.10]{kedlaya-swan1}.
For positive integers $a,b$ with $\gcd(a,b) = 1$,
let $x^{-a} \sim y^{-b}$ denote the exceptional divisor
of the blowup of the ideal sheaf on $\PP^1_k \times \PP^1_k$
concentrated at $(\infty, \infty)$ generated by $x^{-a}, y^{-b}$.
We extend this notation to the case $(a,b) = (1,0), (0,1)$, meaning the
divisors on $\PP^1_k \times \PP^1_k$ cut out by $y^{-1}, x^{-1}$,
respectively.

For $r \in \QQ_{\geq 0}$, write $r = b/a$ in lowest
terms and write $x^{-1} \sim y^{-r}$ for $x^{-a} \sim y^{-b}$.
Along $x^{-1} \sim y^{-r}$, the Swan conductor in its natural normalization
is $a+b$, which behaves erratically as $r$ varies.
However, the normalization with respect to $y$ is $1 + r$, which
is an affine function of $r$. This behavior will prove to be typical;
see Theorem~\ref{T:div subhar}.
\end{example}

\section{$\nabla$-modules on polyannuli}
\label{sec:polyannuli}

The easiest setting in which to study the variation of differential
highest breaks and Swan conductors is on polyannuli, or more conveniently on
the generalized polyannuli of \cite[\S 4]{kedlaya-part3}. 
Using some analysis of differential modules on such spaces carried out
in \cite{kedlaya-xiao} (jointly with Liang Xiao), 
we obtain a strong result on the variation of 
differential Swan conductors (Theorem~\ref{T:convex variation}).
In fact, all results in this section should be considered to be joint
work with Liang Xiao, as explained in Remark~\ref{R:joint}.

\subsection{Convex functions}

We need some basic definitions and theorems about convex functions
from \cite[\S 2]{kedlaya-part3} and \cite[\S 3]{kedlaya-xiao}.
For stronger results along these lines, see \cite{kedlaya-tynan}.

\begin{defn}
Let $C$ be a convex subset of $\RR^n$. A function $f: C \to \RR$ is
\emph{convex} if for all $x,y \in C$ and $t \in [0,1]$,
\[
tf(x) + (1-t)f(y) \geq f(tx + (1-t)y);
\]
such a function is continuous on the interior of $C$.
\end{defn}

\begin{defn}
An \emph{affine functional} on $\RR^n$ is a function
$\lambda: \RR^n \to \RR$ of the form
$\lambda(x) = a_1 x_1 + \cdots + a_n x_n + b$ for some
$a_1, \dots, a_n, b \in \RR$. We say $\lambda$ is \emph{transintegral}
if $a_1, \dots, a_n\in \ZZ$ and \emph{integral} if also $b \in \ZZ$.
\end{defn}

\begin{defn}
A subset $C$ of $\RR^n$ is 
\emph{(trans)rational polyhedral}, or \emph{(T)RP},
if there exist (trans)integral affine
functionals $\lambda_1, \dots, \lambda_m$ such that
\[
C = \{x \in \RR^n: \lambda_i(x) \geq 0 \qquad (i=1, \dots, m)\}.
\]
In particular, any TRP set is convex and closed (but not necessarily bounded).
\end{defn}

\begin{defn}
Let $C$ be a (T)RP subset of $\RR^n$. 
A function $f: C \to \RR$ is \emph{polyhedral} 
if there exist
affine functionals $\lambda_1, \dots, \lambda_m$ such that
\[
f(x) = \sup_i \{\lambda_i(x) \} \qquad (x \in C).
\]
Such a function is continuous and convex.
We say that $f$ is \emph{(trans)integral polyhedral}
if the $\lambda_i$ can be taken to be (trans)integral.
\end{defn}

The following result is \cite[Theorem~2.4.2]{kedlaya-part3}.
\begin{theorem} \label{T:integral polyhedral}
Let $C$ be a bounded RP subset of $\RR^n$. Then a continuous
convex function $f: C \to \RR$
is integral polyhedral if and only if 
\begin{equation} \label{eq:integral values}
f(x) \in \ZZ + \ZZ x_1 + \cdots + \ZZ x_n \qquad (x \in C \cap \QQ^n).
\end{equation}
\end{theorem}

The following result is \cite[Theorem~3.2.4]{kedlaya-xiao}.
\begin{theorem} \label{T:trp}
Let $C$ be a TRP subset of $\RR^n$. Then a function $f: C \to \RR$
is transintegral polyhedral if and only if its restriction to
the intersection of $C$ with every 
$1$-dimensional TRP subset of $\RR^n$ is transintegral polyhedral.
\end{theorem}

\subsection{Generalized polyannuli}

We set notation as in \cite[\S 4]{kedlaya-part3}.

\begin{notation}
For $* = (*_1, \dots, *_n)$ and
$J = (J_1, \dots, J_n)$, we interpret $*^J$ to mean $*_1^{J_1} \cdots *_n^{J_n}$.
\end{notation}

\begin{defn}
A subset $S$ of $(0, +\infty)^n$ is \emph{log-(T)RP} if
$\log(S) \subseteq \RR^n$ is a (trans)rational polyhedral set. 
We say $S$ is \emph{ind-log-(T)RP} if it
is a union of an increasing sequence of log-(T)RP sets; for instance, any
open subset of $(0, +\infty)^n$ is covered by ind-log-RP subsets.
\end{defn}

\begin{defn}
For $S$ an ind-log-TRP set, let $A_K(S)$ be the subspace of the rigid analytic
$n$-space with coordinates $t_1, \dots, t_n$ defined by the condition
\[
(|t_1|, \dots, |t_n|) \in S.
\]
The elements of $\Gamma(A_K(S), \calO)$ can be represented by formal
series $\sum_{J \in \ZZ^n} c_J t^J$; for $R = (R_1, \dots, R_n) \in S$,
write $|\cdot|_R$ for the $R$-Gauss norm
\[
\left| \sum c_J t^J \right|_R = \sup_J \{|c_J| R^J\}.
\]
\end{defn}

\begin{lemma} \label{L:hadamard}
Let $S \subseteq (0, +\infty)^n$ be an ind-log-RP subset.
For $A,B \in S$ and $c \in [0,1]$, put $R = A^c B^{1-c}$; that is,
$r_i = a_i^c b_i^{1-c}$ for $i=1, \dots, n$.
Then for any $f \in \Gamma(A_K(S), \calO)$,
\[
|f|_R \leq |f|_A^c |f|_B^{1-c}.
\]
\end{lemma}
\begin{proof}
See \cite[Lemma~3.1.6(b)]{kedlaya-part1} or
\cite[Lemma~4.1.7]{kedlaya-part3}.
\end{proof}
The following corollary is loosely analogous to Lemma~\ref{L:intersection2}.
\begin{cor} \label{C:log-convex hull}
Let $S_1, S_2$ be log-RP subsets of $(0, +\infty)^n$ 
with nonempty intersection, and let 
\[
S  = \{A^c B^{1-c}: A \in S_1, B \in S_2, c \in [0,1]\}
\]
be the log-convex hull of $S_1, S_2$. Then inside
$\Gamma(A_K(S_1\cap S_2), \calO)$, we have
\[
\Gamma(A_K(S_1), \calO) \cap \Gamma(A_K(S_2), \calO)
= \Gamma(A_K(S), \calO).
\]
\end{cor}
\begin{defn} \label{D:gen rad}
Let $S$ be a log-TRP set, and let $\calE$ be a $\nabla$-module on
$A_K(S)$ relative to $K_0$. For $R \in S$, let $F_R$ be the completion of $\Frac
\Gamma(A_K(S), \calO)$ under $|\cdot|_R$, viewed as a differential field
of order $m+n$ with respect to 
\[
\del_1, \dots, \del_{m+n} = \frac{\del}{\del u_1}, \dots, \frac{\del}{\del u_m},
\frac{\del}{\del t_1}, \dots, \frac{\del}{\del t_n}.
\]
Put
\[
\calE_R = \Gamma(A_K(S), \calE) \otimes_{\Gamma(A_K(S),\calO)} F_R,
\]
viewed as a differential module over $F_R$. Let $S(\calE,R)$ be 
the multiset of reciprocals of the scale multiset of $\calE_R$.
Let $T(\calE,R)$ be the least element of $S(\calE,R)$, i.e., the
reciprocal of the scale of $\calE_R$.
These constructions are stable under shrinking $S$, so they make sense
even if $S$ is only ind-log-TRP.
\end{defn}

The main result we need about differential modules on generalized polyannuli
is \cite[Theorem~3.3.8]{kedlaya-xiao}.
\begin{theorem} \label{T:polyannuli convex}
Let $S$ be an ind-log-TRP subset of $(0, +\infty)^n$, 
and let $\calE$ be a $\nabla$-module of rank $d$ on $A_K(S)$
relative to $K_0$.
For $r \in -\log S$, write $S(\calE,e^{-r}) = \{e^{-f_1(\calE,r)},
\dots, e^{-f_d(\calE,r)}\}$
with $f_1(\calE,r) \geq \cdots \geq f_d(\calE,r)$,
 and put $F_i(\calE,r) = f_1(\calE,r) + \cdots
+ f_i(\calE,r)$. Then the following hold for $i=1,\dots,d$.
\begin{enumerate}
\item[(a)] (Continuity)
The functions $f_i(\calE,r)$ and $F_i(\calE,r)$ are continuous.
\item[(b)] (Convexity)
The function $F_i(\calE,r)$ is convex.
\item[(c)] (Polyhedrality)
The functions $d! F_i(\calE,r)$ and $F_d(\calE,r)$ are transintegral
polyhedral on any TRP subset of $-\log S$.
\end{enumerate}
\end{theorem}

\subsection{Solvable modules on polyannuli}

\begin{hypothesis} \label{H:solvable}
Throughout this subsection,
let $S$ be an ind-log-RP set of the form
$\{R^c: R \in T, c \in (0,1]\}$ for $T$ a log-RP set.
Let $\calE$ be a $\nabla$-module of rank $d$ on $A_K(S)$ relative to $K_0$. 
\end{hypothesis}

\begin{defn}
We say that $\calE$ is \emph{solvable at $1$} if for each $R \in T$,
we have $T(\calE, R^c) \to 1$ as $c \to 0^+$. 
In case $-\log T$ is bounded, it is the log-convex hull of its vertices,
which we write as $-\log R_1, \dots, -\log R_l$
for suitable $R_1,\dots,R_l \in T$.
Then by the convexity
in Theorem~\ref{T:polyannuli convex} (or an argument 
using Lemma~\ref{L:hadamard}, as in 
\cite[Proposition~4.2.6]{kedlaya-part3}), to check solvability,
it suffices to do so for $R = R_1,\dots,R_l$.
\end{defn}

\begin{theorem} \label{T:convex variation1}
Suppose that $\calE$ is solvable at $1$. Then
there exist a constant $\epsilon \in (0,1]$ and functions
$b_1(\calE, r),\dots,b_d(\calE,r)$ on $-\log T$ such that
\begin{equation} \label{eq:stabilize}
S(\calE, e^{-cr}) = \{e^{-cb_1(\calE,r)},\dots, e^{-cb_d(\calE,r)}\}
\qquad (c \in (0, \epsilon]; r \in -\log T).
\end{equation}
Moreover, the functions $d! (b_1(\calE,r) + \cdots + b_i(\calE,r))$ 
and $b_1(\calE,r) + \cdots + 
b_d(\calE,r)$ are convex and integral polyhedral.
\end{theorem}
\begin{proof}
Extend $d! F_i(\calE,r)$ and $F_d(\calE,r)$ to
$U = \{cr: r \in -\log T, c \in [0,1]\}$ by forcing them to take
the value $0$ at $r = 0$.
By Theorem~\ref{T:polyannuli convex}, the functions are convex
and transintegral polyhedral on any 1-dimensional
TRP subset of $U$ not containing 0.
We claim that the same is true for a 1-dimensional
TRP subset of $U$ passing through 0;
the missing assertion is that the functions are affine in a neighborhood of
0 on any line with rational slopes. This holds by virtue of
\cite[Theorem~2.7.2]{kedlaya-swan1}.

We may thus apply Theorem~\ref{T:trp} to deduce that
$d! F_i(\calE,r)$ and $F_d(\calE,r)$ are transintegral polyhedral on
$U$. This gives the existence of $\epsilon$ and the $b_i$,
as well as the convexity and polyhedrality of
$d! (b_1(\calE,r) + \cdots + b_i(\calE,r))$ and $b_1(\calE,r) + \cdots + 
b_d(\calE,r)$. We may deduce the integral polyhedrality by then
applying Theorem~\ref{T:integral polyhedral}.
\end{proof}

\begin{remark} \label{R:joint}
In the original version of this
paper, the results of this section
were only proved assuming that $\calE$ admits a Frobenius
structure. This was needed to ensure the existence of $\epsilon$
such that \eqref{eq:stabilize} holds, as we were unable to prove this
otherwise.
It is the more careful analysis of differential modules on $p$-adic
polyannuli in the joint paper \cite{kedlaya-xiao} with Liang Xiao that makes
the stronger result possible; consequently, we consider all results
in this section to be joint work with Xiao.
\end{remark}

\begin{remark}
One can also obtain a decomposition theorem in case one of the functions
$b_1(\calE,r) + \cdots + b_i(\calE,r)$ is affine, by using
\cite[Theorem~3.4.2]{kedlaya-xiao}. However, the conclusion
will only hold on the interior of $S$.
\end{remark}

\subsection{Geometric interpretation}

We now interpret the previous calculation in terms of Swan conductors.

\begin{hypothesis} \label{H:variation}
Let $\overline{X}$ be a smooth irreducible $k$-variety.
Let $D_1,\dots,D_n$ be smooth irreducible divisors on $\overline{X}$
meeting transversely at a closed
point $x$. Choose local coordinates $t_1,\dots,t_n$
at $x$ such that $t_i$ vanishes along $D_i$.
Put $D = D_1 \cup \cdots \cup D_n$
and $X = \overline{X} \setminus D$. 
Let $\calE$ be an isocrystal of rank $d$ on $X$ overconvergent along $D$.
\end{hypothesis}

We next state an analogue of Theorem~\ref{T:polyannuli convex},
with a similar proof.
\begin{hypothesis} \label{H:variation2}
Assume Hypothesis~\ref{H:variation},
but suppose further that $\overline{X}$ is affine and that the common zero
locus of $t_1,\dots,t_n$ on $\overline{X}$ consists solely of $x$.
Let $P$ be a smooth affine irreducible formal scheme over
$\Spf \gotho_K$ with $P_k \cong \overline{X}$, and 
choose $\tilde{t}_1, \dots, \tilde{t}_n \in \Gamma(P, \calO)$ lifting
$t_1,\dots,t_n$. Realize $\calE$ as a $\nabla$-module
relative to $K_0$
on the space 
\[
\{y \in P_K: \epsilon \leq |\tilde{t}_i(y)| \leq 1 \qquad (i=1,\dots,n)\}.
\]
For $R \in [\epsilon,1]^n$, 
let $|\cdot|_R$ be the supremum norm on the space
\[
\{y \in P_K: |\tilde{t}_i(y)| = R_i \qquad (i=1,\dots,n)\},
\]
then define $S(\calE,R)$ as in Definition~\ref{D:gen rad}.
\end{hypothesis}

\begin{theorem} \label{T:polyannuli convex2}
Under Hypothesis~\ref{H:variation2},
for $r \in [0, -\log \epsilon]^n$,
write $S(\calE,e^{-r}) = \{e^{-f_1(\calE,r)},
\dots, e^{-f_d(\calE,r)}\}$
with $f_1(\calE,r) \geq \cdots \geq f_d(\calE,r)$,
 and put $F_i(\calE,r) = f_1(\calE,r) + \cdots
+ f_i(\calE,r)$. 
Then the following hold for $i=1,\dots,d$.
\begin{enumerate}
\item[(a)] (Continuity)
The functions $f_i(\calE,r)$ and $F_i(\calE,r)$ are continuous.
\item[(b)] (Convexity)
The function $F_i(\calE,r)$ is convex.
\item[(c)] (Polyhedrality)
The functions $d! F_i(\calE,r)$ and $F_d(\calE,r)$ are transintegral
polyhedral on $[0, -\log \epsilon]^n$.
\end{enumerate}
\end{theorem}
\begin{proof}
By Theorem~\ref{T:trp}, it suffices to check that $d! F_i(\calE,r)$
and $F_d(\calE,r)$ are transintegral polyhedral on any transrational
line segment $L$ contained in $[0, -\log \epsilon]^n$.
Let $L$ be such a segment parallel to the vector $a = (a_1,\dots,a_n)
\in \ZZ^n$ with $\gcd(a_1,\dots,a_n) = 1$.
For any indices $i \neq j$, we may replace $a_i$ by $a_i \pm a_j$
by blowing up or down on $\overline{X}$; we may thus reduce to 
the case where $a = (1,0,\dots,0)$.

We now reduce the problem to a corresponding problem in dimension
1, using an analogue of the generic fibre construction of
Definition~\ref{D:generic fibre}.
(Here it is important that we are working relative to a subfield
$K_0$ of $K$; see Remark~\ref{R:relative}.)
Let $R$ be the Fr\'echet
completion of 
\[
\Gamma(P, \calO) \otimes_{K[\tilde{t}_2, \dots, \tilde{t}_n]} 
K(\tilde{t}_2, \dots, \tilde{t}_n)
\]
for the norms $|\cdot|_{e^{-r}}$ for $r \in L$.
Let $K'$ be the integral closure in $R$ of the completion
of $K(\tilde{t}_2, \dots, \tilde{t}_n)$ 
for the $(e^{-r_2},\dots,e^{-r_n})$-Gauss
norm for some $r \in L$. (This does not depend on $r$ because
the elements of $L$ only differ in their first components.)
Then $R$ is an affinoid algebra over $K'$ in which $|\tilde{t}_1|_R \leq 1$.
Moreover, if we put $Y = \Maxspec R$,
then the subspace $\{y \in Y: |\tilde{t}_1(y)| < 1\}$
is isomorphic to the open unit disc over $K'$ with coordinate $\tilde{t}_1$.

For some $\delta > 0$, $\calE$ gives
rise to a $\nabla$-module $\calF$ relative to $K_0$ on the space
$\{y \in Y: \delta \leq |\tilde{t}_1(y)| \leq 1\}$.
On this space, we may carry out a computation analogous
to \cite[Theorem~2.4.4]{kedlaya-xiao} to deduce that
$d! F_i(\calE,r)$ and $F_d(\calE,r)$ are transintegral polyhedral on 
$L$.
\end{proof}

This in turn leads to an analogue of Theorem~\ref{T:convex variation1}.
\begin{defn}
Under Hypothesis~\ref{H:variation}, 
let $T$ be the simplex $\{(r_1,\dots,r_n) \in [0,1]^n: r_1 + \cdots 
+ r_n = 1\}$.
For $r \in T$, define the valuation $v_r$ on $k(\overline{X})$ to be the 
restriction from the $(r_1,\dots,r_n)$-Gauss valuation on 
$\Frac k \llbracket t_1,\dots,t_n \rrbracket$; this valuation is divisorial
if and only if $r \in T \cap \QQ^n$.
\end{defn}

\begin{theorem} \label{T:convex variation3}
Under Hypothesis~\ref{H:variation}, 
there exist $\epsilon \in (0,1]$ and functions
$b_1(\calE, r),\dots,b_d(\calE,r)$ on $-\log T$ such that
\begin{equation} \label{eq:stabilize3}
S(\calE, e^{-cr}) = \{e^{-cb_1(\calE,r)},\dots, e^{-cb_d(\calE,r)}\}
\qquad (c \in (0, \epsilon]; r \in -\log T).
\end{equation}
Moreover, for $i=1,\dots,d$,
the functions $d! (b_1(\calE,r) + \cdots + b_i(\calE,r))$ 
and $b_1(\calE,r) + \cdots + 
b_d(\calE,r)$ are convex and integral polyhedral.
\end{theorem}
\begin{proof}
Given Hypothesis~\ref{H:variation}, we can achieve
Hypothesis~\ref{H:variation2} by shrinking $\overline{X}$ to a suitable
open affine neighborhood of $x$. We then deduce the claim by
replacing Theorem~\ref{T:polyannuli convex} with
Theorem~\ref{T:polyannuli convex2}
in the proof of Theorem~\ref{T:convex variation1}.
(The analogue of the solvability hypothesis is the hypothesis that
$\calE$ arises from an isocrystal on $X$ overconvergent along $D$.)
\end{proof}

Reinterpreting Theorem~\ref{T:convex variation3}
in terms of Swan conductors gives the following.
\begin{theorem} \label{T:convex variation}
Under Hypothesis~\ref{H:variation}, 
for $i=1, \dots, d$ and 
$r \in T \cap \QQ^n$, let
$b_i(\calE,r)$ denote the $i$-th largest differential ramification break
of $\calE$ along $v_r$, normalized with respect to $t_1 \cdots t_n$.
Put $B_i(\calE,r) = b_1(\calE,r) + \cdots + b_i(\calE,r)$.
Then the functions $d! B_i(\calE,r)$ and $B_d(\calE,r)$ are
continuous, convex, and integral polyhedral on $T$.
\end{theorem}
\begin{proof}
It suffices to check that the quantities $b_i(\calE,r)$ as defined in the
statement of the theorem coincide with those defined in 
Theorem~\ref{T:convex variation3}, as then that theorem implies
the claims. For this, impose
Hypothesis~\ref{H:variation2} 
as in the proof of Theorem~\ref{T:convex variation3}.
We may blow up or down on $\overline{X}$ as needed to reduce
the claim for general $r \in T \cap \QQ^n$
 to the claim for $r = (1,0,\dots,0)$,
in which case it is evident from Definition~\ref{D:extend Swan}.
\end{proof}

\begin{remark}
It may be possible to use Theorem~\ref{T:convex variation}
to give a new proof of local semistable reduction of
overconvergent $F$-isocrystals at monomial valuations
\cite[Theorem~6.3.1]{kedlaya-part3}. Such an argument would likely
give some results without having to assume that $K$ is discretely valued,
as is necessary in \cite{kedlaya-part3} due to the use of Frobenius
slope filtrations.
\end{remark}

\section{Variation near a surface divisor}

We now make a more careful study of 
the variation of differential Swan conductors on a surface,
in the vicinity of a single irreducible divisor. 

\subsection{A raw calculation}
\label{subsec:variation}

\begin{hypothesis}
Throughout this subsection:
\begin{itemize}
\item
assume that $k = k_0$ is algebraically closed;
\item
let $P$ be a smooth irreducible formal scheme over $\Spf \gotho_K$,
such that $Z = P_k$ is an open dense subscheme of a
curve of genus $g = g(Z)$;
\item
let $U$ denote an open dense affine subscheme of $Z$;
\item
let $L$ be the completion of $\Frac \Gamma(U, \calO)$ for the supremum
norm on $]U[_P$ (this does not depend on $U$);
\item
let $Y$ denote a fringed relative annulus over $]U[_P$ 
(as in Remark~\ref{R:proper fringed});
\item
let $\calE$ be a $\nabla$-module on $Y$ of rank $d$,
which is solvable at $1$.
\end{itemize}
\end{hypothesis}

\begin{defn} \label{D:local variation}
Choose $\epsilon \in (0,1)$  as in Definition~\ref{D:fringed}.
For a closed point $z \in Z$, choose a local uniformizer $\overline{x} \in \calO_{Z,z}$ 
of $Z$ at $z$. Choose a lift $x$ of $\overline{x}$ to
$\Gamma(Q, \calO)$ for some open dense formal subscheme $Q$ of $P$
containing $z$.
This choice gives an isomorphism $]z[_P \times A_K[0,1) \cong A_K[0,1)^2$;
for each $\rho \in (\epsilon,1)$,
for $r \in (0,+\infty)$ in some neighborhood of 0 (depending on $\rho$),
we may then compute $S(\calE, (\rho^r, \rho))$ and $T(\calE,(\rho^r,\rho))$
in the sense of Definition~\ref{D:gen rad}. To indicate the dependence on $z$,
we write these as $S(\calE,z,(\rho^r,\rho))$ and $T(\calE,z,(\rho^r,\rho))$.
We extend the definitions to $r = 0$ by putting $S(\calE,(1,\rho))
= S(\calF_\rho)$ and $T(\calE,(1,\rho)) = T(\calF_\rho)$, 
for $\calF$ the generic fibre of $\calE$.

Note that we have omitted the dependence on $\overline{x}$ and $x$ from the
notation. That is because we are only interested here
in behavior as $r$ approaches $0$, 
in which limit the choice of $x$ (or $\overline{x}$)
does not matter. To see this,
suppose $x' \in \Gamma(Q, \calO)$ also lifts a local uniformizer of $Z$ at $z$.
We can then write $x' = \sum_{i=0}^\infty
c_i x^i$ with $|c_0| < 1$, $|c_1| = 1$, and $|c_i| \leq 1$ for $i > 1$. If
$\rho^r \geq |c_0|$, then $|x|_{\rho^r} = |x'|_{\rho^r}$.
Hence for each $\rho \in (0,1)$, for $r \in (0, +\infty)$ sufficiently close to
0, the quantities $S(\calE,z,(\rho^r,\rho))$ and $T(\calE,z,(\rho^r,\rho))$
are the same regardless of whether we use $x$ or $x'$ to define the isomorphism
$]z[_P \times A_K[0,1) \cong A_K[0,1)^2$. (The definitions for $r=0$
visibly do not depend on this choice.)
\end{defn}

\begin{prop} \label{P:div subhar1}
We can choose a subset $R$ of $(0,1)$ of the form
$(\epsilon,1) \setminus R'$, where $R'$ is a set with discrete limit points,
such that the following statements hold.
\begin{enumerate}
\item[(a)]
For each $z \in Z$ and $\rho \in R$,
there exist affine functions $b_1(\rho,r), \dots, b_d(\rho,r)$ on 
$[0,a]$, for some $a>0$, such that
\[
S(\calE,z,(\rho^r,\rho)) = \{\rho^{b_1(\rho,r)}, \dots,
\rho^{b_d(\rho,r)}\} \qquad
(r \in [0,a]).
\]
\item[(b)]
For $z \in Z$ and $\rho \in R$, put
\[
f(\rho,z,r) = \sum_{\alpha \in S(\calE,z,(\rho^r,\rho))} 
\log_\rho \alpha
\]
and write $f'(\rho,z)$ for the right slope of $f(\rho,z,r)$ at $r=0$.
Then there exist $\ell \in \{0, 1,\dots,d\}$ 
(independent of $\rho$)
and a choice of the open dense subscheme $U$ of $Z$ 
(dependent on $\rho$)
such that $f'(\rho,z) = -\ell$ for all $z \in U$.
\item[(c)]
Assume that $Z$ is proper.
With notation as in (b), we have
\begin{equation} \label{eq:div subhar1}
\sum_{z \in Z} (f'(\rho,z) + \ell) \geq (2-2g(Z))\ell.
\end{equation}
\end{enumerate}
\end{prop}
\begin{proof}
There is no harm in shrinking $U$ or $Y$, so we may assume that $\calE$
is indecomposable and remains so upon further shrinking of $U$ or $Y$.
We may also assume that we can choose 
$u \in (\Frac \Gamma(P,\calO)) \cap \Gamma(]U[, \calO)$ 
such that $du$ freely generates
$\Omega^1_{P/\gotho_K}$ over $]U[_P$; 
put $\del_1, \del_2 = \frac{\del}{\del u},
\frac{\del}{\del t}$. 
Let $s_{i,1}(\rho,z,r) \leq \dots \leq s_{i,d}(\rho,z,r)$ 
be the reciprocals of the elements
(counted with multiplicity) of the scale multiset of $\del_i$ 
on $\calE_{(\rho^r,\rho)}$ in the bidisc $]z[_P \times A_K[0,1)$.
Choose $\epsilon \in (0,1)$
as in Definition~\ref{D:fringed},
and also satisfying $T(\calE,\rho) = \rho^b$ for all $\rho \in (\epsilon,1)$,
where $b$ is the highest break of $\calE$.

Set notation as in the proof of Proposition~\ref{P:hard case}.
Choose $i$ such that $\del_i$ is eventually dominant for $\calE$.
Then for all
$\rho \in I$ except for a discrete subset $R'_I$,
we can read off the $s_{i,j}(\rho,z,r)$  from the Newton polygon
of the twisted polynomial $Q$: for $r=0$ they are all equal to
$T(\calE,\rho) = \rho^b$ by the conclusion of Proposition~\ref{P:hard case},
so for $r$ close to zero, we do not cross the threshold set by
\cite[Proposition~1.1.9]{kedlaya-swan1}
for reading off scales from slopes of the Newton polygon.
We deduce that for each $\rho \in I \setminus R'_I$, 
we can choose $a>0$ such that 
each function $r \mapsto \log s_{i,j}(\rho,z,r)$
is affine for $r \in [0,a]$.
(That is because these functions measure the slopes of a Newton polygon
whose vertices vary linearly in $r$ when $r$ is 
sufficiently close to $0$.)
In particular, we may apply \cite[Proposition~1.1.9]{kedlaya-swan1}
or \cite[Theorem~2.3.5]{kedlaya-xiao}
to perform a simultaneous scale decomposition of $\calE$ for $\del_i$
over $A_K(S)$, for $S = \{(\rho^r,\rho): r \in (0,a)\}$.
Let $m_{i,j}(\rho,z)$ be the right slope of $\log_\rho s_{i,j}(\rho,z,r)$
at $r=0$.

Consider the case $i=2$. 
Given $h \in \{0, \dots, d-1\}$,
write $a_h = \sum_j f_j t^j$; by the choice of $\rho$,
there is a unique $j = j(h)$ which minimizes $|f_j|_L \rho^j$.
Choose $\lambda_j \in K^\times$ with $|\lambda_j| = |f_j|_L$;
then if we shrink $U$ so as not to meet the zero locus of
the reduction of $\lambda_{j(h)}^{-1} f_{j(h)}$ for any $h$,
then the $m_{2,j}(\rho,z)$ vanish for all $z \in U$ by
Lemma~\ref{L:same norm}. Also,
$\sum_{j=1}^d m_{2,j}(\rho,z)$ equals the order of vanishing at $z$
of the reduction of $f_{j(0)} \lambda_{j(0)}^{-1}$, so its sum over
$z \in Z$ equals 0 if $Z$ is proper.

Consider the case $i=1$.
Rotate as in the proof of Proposition~\ref{P:hard case},
i.e., first pull back along $t \mapsto t^{p^N}$ for a large
integer $N$, then along $u \mapsto u + t$. 
The effect of the first step is to pull back the action of $\del_1$ unchanged,
while replacing the action of $\del_2$ by the pullback action of $\del_2$ times
$p^N t^{p^N - 1}$.
The effect of the second step is to pull back the action of $\del_1$ unchanged,
while replacing the action of $\del_2$ by the pullback action of $\del_2
+ \del_1$.
Consequently, after rotation with $N$ sufficiently large,
for $r$ sufficiently small
the reciprocals of the scale multiset of $\del_2$ on $\calE_{(\rho^r,\rho)}$
in the 
bidisc $]z[_P \times A_K[0,1)$ consist of
\[
\rho^{rc+r-1} s_{1,1}(\rho,z,r), 
\dots, \rho^{rc+r-1} s_{1,d}(\rho,z,r),
\]
where
$c$ equals the order of vanishing of the differential $du$ on
$Z$ at the point $z$. (The factor $\rho^{r-1}$ comes from the change of
normalization in measuring the scale of $\del_2$ rather than $\del_1$.
The factor $\rho^{rc}$ comes from the fact that for $x$ a local parameter
of $Z$ at $z$, $\del_1$ equals $x^{-c}\,\frac{\partial}{\partial x}$ 
times a unit
in $\calO_{Z,z}$.)
In particular, each $m_{1,j}(\rho,z)$ equals
$-1$ for all but finitely many $z \in Z$,
and the sum of $d + \sum_{j=1}^d m_{1,j}(\rho,z)$ over all $z \in Z$
equals $(2-2g(Z))d$ if $Z$ is proper.

If $\del_i$ is eventually
dominant for only one $i$, then for each $z$,
we have $S(\calE,z,(\rho^r,\rho)) = \{s_{i,1}(\rho,z,r), \dots, 
s_{i,d}(\rho,z,r)\}$
for $r$ close to zero, so all the desired results follow with
\[
\ell = \begin{cases} d  & i=1 \\
0 & i=2.
\end{cases}
\]
(The excluded set $R'$ consists of those $\rho$ not appearing in $I
\setminus R'_I$ for any $I$; the only limit points of this set are
those $\rho$ for which $T(\calE,\rho) = |p|^{p^{-m}/(p-1)}$ for some
$m \in \ZZ$.) 
Assume hereafter that both $\del_1,\del_2$ are eventually dominant;
we will again prove the claims with $\ell=0$.

To deduce (a), note that for each $z$, we can choose $a > 0$ such that for
$S = \{(\rho^r,\rho): \rho \in I, r \in (0,a)\}$,
we obtain a simultaneous scale decomposition of $\calE$
for both $\del_1$ and $\del_2$ over $A_K(S)$.
(Compare \cite[Theorem~3.4.2]{kedlaya-xiao}.)

To deduce (b), note that by shrinking $U$, we can ensure that
for all $z \in U$, 
$s_{2,j}(\rho,z,r)$ is constant for small $r$,
and $m_{2,j}(\rho,z) = 0$; by rotation,
we can also ensure that for all $z \in U$, $m_{1,j}(\rho,z) = -1$.
Consequently, for $z \in U$, for $r$ close to 0, $S(\calE,z,(\rho^r,\rho))$
consists of $T(\calE,\rho)$ with multiplicity $d$.

To deduce (c) if $Z$ is proper, note that 
$f'(\rho,z) \geq \sum_{j=1}^d m_{2,j}(\rho,z)$, 
and summing the right side over
$z$ yields 0.
\end{proof}

\begin{remark}
One might like to prove Proposition~\ref{P:div subhar1} directly by reading
off the Swan conductor from a twisted polynomial, without having
to decompose into indecomposables. There are two reasons
why this will not work. One is the fact that different derivations may
be dominant on different components of the break decomposition.
The other is 
the limitation on slopes in \cite[Proposition~1.1.9]{kedlaya-swan1}:
the presence of some $\lambda$ in a radius multiset masks the presence
of any $\lambda' > \lambda^{1/p}$ when viewing Newton polygons.
By working in the indecomposable case, we fail to encounter this masking for
$r$ sufficiently small because we have a uniform break at $r=0$.
\end{remark}

\begin{remark}
The arguments in \cite[\S 2.4]{kedlaya-xiao} are in a similar spirit.
Using ideas from there, it should be possible to remove the restriction
to the set $R$ in Proposition~\ref{P:div subhar1}.
\end{remark}

\subsection{Subharmonicity}

We now obtain a subharmonicity theorem for differential Swan 
conductors on a surface.

\begin{hypothesis} \label{H:divisorial3}
Assume that $k = k_0$ is algebraically closed.
Let $\Xbar$ be a smooth irreducible projective surface over $k$,
let $Z$ be a smooth irreducible divisor on $\Xbar$,
and let $v_0$ be the divisorial valuation on $k(\Xbar)$ measuring
order of vanishing along $Z$.
Let $W$ be a divisor not containing $Z$,
and put $Y = \Xbar \setminus W$; note that $Y \cap Z$ is open dense in $Z$.
Let $X$ be an open dense subscheme of $Y$, and let
$\calE$ be an isocrystal of rank $d$ 
on $X$ overconvergent along $Y \setminus X$.
\end{hypothesis}

\begin{defn} \label{D:restrict to annulus}
Let $P$ be a smooth formal scheme over $\Spf \gotho_K$ with 
special fibre $Z \cap Y$. As in Definition~\ref{D:geometric Swan},
for any open affine subscheme $Z_0$ of $Z \cap Y$, 
we obtain from $\calE$ a
$\nabla$-module on a fringed relative annulus over $]U[_P$,
for some open dense subscheme $U$ of $Z_0$. Moreover,
any two such $\nabla$-modules so obtained become isomorphic on
a suitably small fringed relative annulus, so the construction glues to
give a $\nabla$-module on a fringed relative annulus 
over $]U[_K$, for some open dense
subscheme $U$ of $Z \cap Y$; 
we will also use the symbol $\calE$ to refer to
this $\nabla$-module.
\end{defn}

\begin{defn} \label{D:local variation3}
Given $z \in Z \cap Y$, choose $x \in \calO_{\Xbar,z}$ whose
zero locus has a single component at $z$,
which is smooth of multiplicity 1 and meets $Z$ transversely.
For $r \in \QQ \cap [0,1]$, 
let $v_r(z; x)$ be the valuation on $k(\Xbar)$ corresponding to the divisor
$x \sim t^r$ 
(in the sense of Example~\ref{ex:Dwork})
on a suitable blowup of $\Xbar$ at $z$, for $t$ a local parameter
of $Z$ at $z$.
If we identify the completion of the local ring $\calO_{\Xbar,z}$ with
$k \llbracket x,t \rrbracket$, then $v_r(z; x)$ is induced by the 
$(r,1)$-Gauss valuation on $k \llbracket x,t \rrbracket$.
The latter valuation is invariant under any continuous automorphism of 
$k \llbracket x,t \rrbracket$ of the form 
$t \mapsto ut, x \mapsto \lambda x + w$
where $u$ is a unit in $k \llbracket x,t \rrbracket$, $\lambda \in k^\times$,
and $w$ belongs to the ideal $(t, x^2)$.
This allows replacing $x$ by
any other $x' \in \calO_{\Xbar,z}$
whose zero locus has a single component at $z$, which is smooth of multiplicity
1 and meets $Z$ transversely. It also allows replacing $t$ by another
local parameter of $Z$ at $z$.
Consequently, those replacements do not affect the definition of $v_r(z; x)$.

Let $b_1(\calE,z,x,r) \geq \cdots \geq b_d(\calE,z,x,r)$ 
and $\Swan_Z(\calE,z,x,r)$ be the differential highest breaks
and Swan conductor of 
$\calE$ along $v_r(z; x)$, normalized with respect to $t$. 
By Theorem~\ref{T:convex variation}, 
the function $r \mapsto b_j(\calE,z,x,r)$ is affine in a neighborhood of $0$.
It thus extends continuously to all $r \in [0,a]$ for some $a>0$.
\end{defn}

\begin{lemma} \label{L:uniform rel}
With notation as in Definition~\ref{D:local variation},
there exist $\epsilon \in (0,1)$
and $a>0$ (depending on $z$) such that 
for $r \in [0,a]$
and $\rho \in (\epsilon,1)$,
$S(\calE,z,(\rho^r,\rho))$ is defined and
\[
S(\calE,z,(\rho^r,\rho)) = \{\rho^{b_1(\calE,z,x,r)},
\dots, \rho^{b_d(\calE,z,x,r)}\}.
\]
\end{lemma}
\begin{proof}
Apply Theorem~\ref{T:convex variation3}.
\end{proof}

The value of $\epsilon$ in Lemma~\ref{L:uniform rel} depends on the
choice of $z$. However, we can use the following argument to make a uniform
choice.
\begin{lemma} \label{L:div subhar1 uniform}
With notation as in Definition~\ref{D:local variation},
suppose that for some $\rho_0 < \rho_1 \in (\epsilon,1)$
and some $c \in \RR$,
$S(\calE,z,(\rho_j^r,\rho_j)) = \{\rho_j^{b_1(\calE,z,x,0)+cr},
\dots, \rho_j^{b_d(\calE,z,x,0)+cr}\}$ for $j=0,1$ and $r \in [0,a]$.
Then there exists $b>0$ such that
for all $\rho \in [\rho_0,1)$ and all $r \in [0,b]$,
$S(\calE,z,(\rho^r,\rho))$ is defined and
\[
S(\calE,z,(\rho^r,\rho)) = \{\rho^{b_1(\calE,z,x,0)+cr},
\dots, \rho^{b_d(\calE,z,x,0)+cr}\}.
\]
\end{lemma}
\begin{proof}
Choose $b \in [0,a]$ so that $S(\calE,z,(\rho^r,\rho))$ is defined for 
all $\rho \in [\rho_0,1)$ and all $r \in [0,b]$.
For $i=1,\dots,d$, and $r \in [0,b] \cap \QQ$,
the function $F_i(\calE,(rs,s))$ is convex for $s \in (0,-\log \rho_0]$
by Theorem~\ref{T:polyannuli convex2},
and extends continuously to $s=0$ with the value 0
because $\calE$ is overconvergent.
On the other hand, it agrees with a linear function at the three
values $s = 0, -\log \rho_1, - \log \rho_0$, so it must be linear
on all of $[0, -\log \rho_0]$.
This proves the claim for $r \in [0,b] \cap \QQ$;
the full claim follows by continuity (Theorem~\ref{T:polyannuli convex2}).
\end{proof}
\begin{cor} \label{C:uniform rel1}
In Lemma~\ref{L:uniform rel}, the value of $\epsilon$ can be chosen
independent of $z \in Z \cap Y$. Moreover, for all but finitely many
$z \in Z \cap Y$, either $b_i(\calE,z,x,r)$ or $b_i(\calE,z,x,r)+r$
(depending on whether $\del_2$ is or is not eventually dominant on the
corresponding component of $\calE$) is constant for $r$ in some
neighborhood of $0$ (depending on $z$).
\end{cor}
\begin{proof}
By the proof of Proposition~\ref{P:div subhar1}, the hypothesis of
Lemma~\ref{L:div subhar1 uniform} holds for all but finitely many
$z \in Z \cap Y$. The assertion is then clear from the proof
of Proposition~\ref{P:div subhar1}.
\end{proof}

\begin{theorem} \label{T:div subhar}
Under Hypothesis~\ref{H:divisorial3}, 
we have the following.
\begin{enumerate}
\item[(a)]
For each $z \in Z \cap Y$, the functions
$b_j(\calE,z,x,r)$ for $j=1,\dots,d$
and $\Swan_Z(\calE,z,x,r)$ are affine in a neighborhood
of $r=0$.
\item[(b)]
Let $\Swan_Z'(\calE,z)$ be the right slope of $\Swan_Z(\calE,z,x,r)$ at $r=0$.
Then there exists $\ell = \ell(\calE,Z) \in \{0, 1,\dots,d\}$ such that
$\Swan_Z'(\calE,z) = -\ell$ for all but finitely many $z \in Z \cap Y$.
\item[(c)]
Assume that $Z \subset Y$.
With notation as in (b), we have
\begin{equation} \label{eq:div subhar}
\sum_{z \in Z} (\Swan_Z'(\calE,z) + \ell) \geq (2-2g(Z))\ell
- Z^2 \Swan(\calE,Z),
\end{equation}
where $g(Z)$ denotes the genus of $Z$, and $Z^2$ denotes the
self-intersection of $Z$ on $\Xbar$.
\end{enumerate}
\end{theorem}
\begin{proof}
We deduce (a) from Lemma~\ref{L:uniform rel},
and (b) from Corollary~\ref{C:uniform rel1}.
For (c), we must account for the fact that we cannot necessarily
choose the local parameter $t$ uniformly for all $z \in Z$. Pick 
$t \in k(\overline{X})$ with $v_0(t) = 1$,
and let $D$ denote the principal divisor defined by $t$; 
then $D \cdot Z = 0$, so $(D -Z)\cdot Z = - Z^2$. 

For $z \in Z$, let $t_z$ be a local parameter for $Z$ at $z$,
and let $c_z$ be the order of vanishing at $z$ of the restriction
of $t/t_z$ to $Z$. Then $c_z$ is equal to the local intersection
multiplicity $((D-Z)\cdot Z)_z$, so $\sum_{z \in Z} c_z = -Z^2$.
Let $x_z \in k(\overline{X})$ cut out a divisor with a single component
at $z$, which is smooth of multiplicity 1 and meets $Z$ transversely.
For $s$ close to 0, the valuation $v_s(z; x_z)$ corresponds to the
divisor $x_z \sim t_z^s$, or $x_z \sim t^r$ with 
$r = s/(1 + sc_z)$. (Again, the notation $\sim$ is used as in
Example~\ref{ex:Dwork}.)

Define $f(\rho,z,r)$ as in Proposition~\ref{P:div subhar1};
by Corollary~\ref{C:uniform rel1}, it is independent of $\rho$ for
$r$ in some neighborhood of $0$ and $\rho$ in some neighborhood of 1,
so we may call the resulting value $f(z,r)$. This quantity is the
Swan conductor along
$x_z \sim t^r$ normalized with respect to $t$; renormalizing with respect
to $t_z$, we obtain
\[
\Swan_Z(\calE,z,x,s) = f(z,r) \frac{v_s(x;z)(t)}{v_s(x;z)(t_z)}
= \frac{s}{r} f(z,r) = (1+sc_z) f(z,r).
\]
Differentiating with respect to $s$ at $r=s=0$ yields
\[
\Swan'_Z(\calE,z) = c_z f(z,0) + f'(z).
\]
We now deduce (c) by summing over $z \in Z$ and invoking 
Proposition~\ref{P:div subhar1}(c).
\end{proof}

\begin{example} \label{exa:subhar}
Here is a typical example where Theorem~\ref{T:div subhar} holds 
with $\ell \neq 0$:
take $Z$ to be the $x$-axis in the $x,t$-plane
$\AAA^2_k \subset \PP^2_k$, take $X = \AAA^2_k \setminus Z$,
and take $\calE$ to be the Dwork isocrystal $\calL_{x t^{-p}}$.
\end{example}

\begin{remark}
As is apparent in the proof of Theorem~\ref{T:div subhar},
the self-intersection number in \eqref{eq:div subhar} is a side effect of
normalizing with respect to a different parameter at each point of $Z$;
it drops out if one normalizes everything with respect to a single function.
\end{remark}

\begin{remark}
It is reasonable to ask whether equality necessarily holds in 
\eqref{eq:div subhar} as long as the ramification breaks along $Z$ are all
nonzero. Unfortunately, the proof of Proposition~\ref{P:div subhar1} does
not suffice to establish this; what is missing is a proof that
if $\frac{\del}{\del t}$ and $\frac{\del}{\del x}$ are both dominant
on $\calE_\rho$, then $\frac{\del}{\del t}$ is dominant 
on $\calE_{(\rho^r,\rho)}$ for $r>0$ small. 
\end{remark}

\subsection{Monotonicity}

We now use some refined results on $p$-adic differential modules
on discs, to gain some further control over differential
Swan conductors. In the original version of this paper, this was done
using results on rigid cohomology to imitate what
one does in the $\ell$-adic setting
(compare Laumon's proof of the semicontinuity theorem
\cite{laumon}); that method was limited to fully
overconvergent $F$-isocrystals, with $K$ discrete.

\begin{defn}
Under Hypothesis~\ref{H:divisorial3}, for $i \in \{1,\dots,d\}$ such
that either $i=d$ or $b_i(\calE,Z) > b_{i+1}(\calE,Z)$, let $\ell_i(\calE,Z)$
be the sum of the ranks of the components of the break decomposition
of $\calE$ contributing to $b_1(\calE,Z) + \cdots + b_i(\calE,Z)$
on which $\del_2$ is not eventually dominant. In particular,
$\ell_d(\calE,Z) = \ell(\calE,Z)$.
\end{defn}

\begin{theorem} \label{T:monoton2}
Assume Hypothesis~\ref{H:divisorial3}.
Suppose that $z \in Z \cap Y$ 
is a smooth point of $Z \cup (\Xbar \setminus X)$.
Let $b'_i(\calE,z)$ be the right slope of $b_i(\calE,z,x,r)$ at $r=0$.
Then for $i \in \{1,\dots,d\}$ such
that either $i=d$ or $b_i(\calE,Z) > b_{i+1}(\calE,Z)$,
we have $b'_1(\calE,z) + \cdots + b'_i(\calE,z) + \ell_i(\calE,Z) \leq 0$,
with equality for all but finitely many $z$.
\end{theorem}
The proof is again by rotation, but this time in the opposite direction
from the arguments of \cite{kedlaya-swan1}:
we use a result about $\del_1$ to prove something about $\del_2$.
\begin{proof}
The equality for all but finitely many $z$ follows from
Corollary~\ref{C:uniform rel1},
so it suffices to check the inequality.
We first treat the case $i = d$.

Take $x,t$ as in Definition~\ref{D:local variation3}.
Because $z$
is a smooth point of $\overline{X} \setminus X$,
we may restrict $\calE$ to a space of the form
$A_{K,x}[0,1) \times A_{K,t}(\epsilon,1)$
for some $\epsilon \in (0,1)$.
By Lemma~\ref{L:uniform rel}, 
we can choose $a>0$ and $\epsilon \in (0,1)$
so that for $r \in (0,a)$
and $\rho \in (\epsilon,1)$,
\[
S(\calE,z,(\rho^r,\rho)) = \{\rho^{b_1(\calE,z,x,r)},
\dots, \rho^{b_d(\calE,z,x,r)}\}.
\]
By Theorem~\ref{T:convex variation}, we can choose $a$ so that
each of $b_1(\calE,z,x,r), \dots, b_d(\calE,z,x,r)$ is affine in $r$
for $r \in [0,a]$.

Pick any $\rho \in (\epsilon,1)$, and let $K_\rho$ be the completion
of $K(t)$ for the $\rho$-Gauss norm. We may then restrict $\calE$
to obtain a $\nabla$-module $\calF$ on 
$A_{K_\rho,x}[0,1)$.
As in the proof of \cite[Theorem~2.4.4]{kedlaya-xiao}, 
for a suitable choice of $a$,
we may decompose $\calF = \oplus_j \calF_j$ 
over $A_{K_\rho,x}(T)$ for $T = \{\rho^r: r \in (0,a)\}$,
so that for 
each $h \in \{1,2\}$, either $\del_h$ is not dominant on
$(\calF_j)_{\rho^r}$ for each $r \in (0,a)$, or $\del_h$ is dominant on
$(\calF_j)_{\rho^r}$ for each $r \in (0,a)$
with scale multiset consisting of a single element. (We abbreviate this 
by saying that $\del_h$ is or is not dominant on $\calF_j$.)

Write the scale of $(\calF_j)_{\rho^r}$ as $\rho^{-\alpha r - \beta}$,
where we write $\alpha = \alpha(\calF_j)$ and $\beta = \beta(\calF_j)$
if it is necessary to disambiguate.
Then
\[
\sum_j (\alpha(\calF_j) r + \beta(\calF_j)) \rank(\calF_j) = \Swan_Z(\calE,z,x,r)
\]
and so
\begin{equation} \label{eq:monotonicity3}
\sum_j \alpha(\calF_j) \rank(\calF_j) = \Swan'_Z(\calE,z).
\end{equation}
Put $\ell(\calF_j) = 0$ if the limit of the scale of $\del_2$ on
$(\calF_j)_{\rho^r}$ as $r \to 0^+$ equals $\rho^{-\beta}$,
and $\ell(\calF_j) = \rank(\calF_j)$ otherwise. Then
\begin{equation} \label{eq:monotonicity4}
\sum_j \ell(\calF_j) = \ell(\calE,Z).
\end{equation}

Let $K_1$ be the completion of $K_{\rho}(u)$ for the $1$-Gauss 
norm. Let $f: A_{K_1}[0,1) \to A_{K_\rho}[0,1)$ be the 
$K_0$-linear map of locally $G$-ringed spaces acting on global
sections via $f^*(x) = x$, $f^*(t) = t(1+ux)$.
This has the effect
\[
dx \mapsto dx, \qquad
dt \mapsto (1+ux)\,dt + ut\,dx.
\]
Writing $\del'_1, \del'_2$ for the actions of $\frac{\partial}{\partial x}$,
$\frac{\partial}{\partial t}$ before pulling back, 
the actions of $\del_1, \del_2$ are given by
\[
\del'_1 + ut \del'_2, \qquad (1+ux) \del'_2.
\]
In particular, the scale of
$\del_1$ on $(f^* \calF_j)_{\rho^r}$
is equal to the greater of the following quantities: 
the scale of $\del'_1$ on $(\calF_j)_{\rho^r}$,
and $\rho^{r}$ times the scale of $\del'_2$ on $(\calF_j)_{\rho^r}$.

Write the scale of $\del_1$ on $(f^* \calF_j)_{\rho^r}$ as $\rho^{-g(r)}$.
Since $\calF$ extends to an affinoid space containing the annulus
$A_{K_\rho,x}(T)$, the proof of 
\cite[Theorem~11.3.2]{kedlaya-course} shows that
each $g(r)$ extends continuously to $[0, a)$, and is affine in
a neighborhood of $r=0$ (as in Theorem~\ref{T:polyannuli convex2}).
Let $m = m(\calF_j)$ be the right slope of $g$ at $r=0$.
From the calculation of the scale of $\del_1$ on $(f^* \calF_j)_{\rho^r}$
above, we have the following.
\begin{itemize}
\item
If $\del_1$ is dominant on $\calF_j$, then
$g(r) = \alpha r + \beta$, so $m = \alpha$.
\item
If $\del_1$ is not dominant on $\calF_j$, then
$\alpha r + \beta > g(r) \geq (\alpha-1) r + \beta$,
so $\alpha > m \geq \alpha-1$.
\end{itemize}

We say that $\calF_j$ is \emph{negligible} if $\alpha = \beta = 0$.
By \cite[Theorem~11.3.2(d)]{kedlaya-course} applied 
on $A_{K_1}[0, 1-\epsilon]$ for some small $\epsilon > 0$, 
we have
\begin{equation} \label{eq:monotonicity1}
\sum_j (m(\calF_j) + 1) \rank(\calF_j) \leq 0,
\end{equation}
provided we take the sum over those $j$ for which $\calF_j$ is not negligible.
For each such $j$, we have the following.
\begin{itemize}
\item
If $\ell(\calF_j) = 0$, then $m \geq \alpha - 1$ whether or not
$\del_1$ is dominant on $\calF_j$, so $m + 1 \geq \alpha$.
\item
If $\ell(\calF_j) = \rank(\calF_j)$, then $\del_2$ cannot be dominant on
$\calF_j$ for $r>0$ small, so $\del_1$ must be dominant on $\calF_j$.
We thus must have $m = \alpha$.
\end{itemize}
In both cases, we have 
\[
(m(\calF_j) + 1) \rank(\calF_j) \geq \alpha(\calF_j) \rank(\calF_j) + \ell(\calF_j),
\]
so by \eqref{eq:monotonicity1} we have
\begin{equation} \label{eq:monotonicity2}
\sum_j (\alpha(\calF_j) \rank(\calF_j) + \ell(\calF_j)) \leq 0
\end{equation}
provided that we only sum over $j$ for which $\calF_j$ is not negligible.
However, the left side of \eqref{eq:monotonicity2} does not change if we
include summands for which $\calF_j$ is negligible (as those have
$\alpha(\calF_j) = \ell(\calF_j) = 0$), so \eqref{eq:monotonicity2}
holds even if we sum over all $j$. By \eqref{eq:monotonicity3}
and \eqref{eq:monotonicity4}, this yields the desired inequality in the case
$i=d$.

We now treat the case where $i <d $ but $b_i(\calE,Z) > b_{i+1}(\calE,Z)$.
Pick a rational number $c/m \in (b_{i+1}(\calE,Z), b_i(\calE,Z))$ with 
denominator
$m$ coprime to $p$. Let $\calF$ be the direct 
sum of the Dwork isocrystals
$\calL_{t^{c/m}}$ 
(in the sense of Definition~\ref{D:dwork iso})
for $t^{c/m}$ running over all of the $m$-th roots of $t^c$.
This isocrystal is initially only defined on an $m$-fold cover of $X$,
but it descends to an overconvergent
isocrystal of rank $m$ such that for $r$ near $0$,
\[
b_1(\calF,z,x,r) = \cdots = b_m(\calF,z,x,r) = \frac{c}{m}
\]
by \cite[Example~3.5.10]{kedlaya-swan1}.
Consequently,
\[
b_{(j-1)m + 1}(\calE \otimes \calF, z,x,r) =
\cdots = b_{jm}(\calE \otimes \calF,z,x,r) = 
\begin{cases}
b_i(\calE,z,x,r) & j \leq i \\
\frac{c}{m} & j > i.
\end{cases}
\]
Thus we may obtain the desired result for $\calE$ by applying the previously
shown case for $\calE \otimes \calF$.
\end{proof}

Equality in Theorem~\ref{T:monoton2} has a special meaning.
\begin{theorem} \label{T:swan drop4}
With notation as in Theorem~\ref{T:monoton2}, suppose that
for each $i \in \{1,\dots,d\}$ such
that $b_i(\calE,Z) > b_{i+1}(\calE,Z)$,
we have $b'_1(\calE,z) + \cdots + b'_i(\calE,z) + \ell_i(\calE,Z) = 0$.
Then the break decomposition of $\calE$ along $Z$ extends over $z$.
\end{theorem}
\begin{proof}
Set notation as in the proof of Theorem~\ref{T:monoton2};
then we must have equality in \eqref{eq:monotonicity1}
for each $i$ such that $b_i(\calE,Z) > b_{i+1}(\calE,Z)$.
By \cite[Theorem~12.2.2]{kedlaya-course},
$f^* \calF$ 
admits a direct sum decomposition over all of $A_{K_\rho}[0,1)$
such that over $A_{K_\rho}(T)$, the $\calF_j$ 
which are grouped into the same summand 
all have the same value of $\beta(\calF_j)$.
Over $A_{K_\rho}(T)$, this decomposition coincides
with the decomposition obtained by pulling back
the break decomposition of $\calE$; in particular, it descends
to a decomposition of $\calF$ itself.

The projectors onto the summands in this decomposition of $\calF$
are horizontal sections of $\calF^\dual \otimes \calF$.
Since these match the projectors over $A_{K_\rho}(T)$
defined by the break decomposition, 
we may apply Lemma~\ref{L:extend break}
to deduce that the break decomposition of $\calE$ along $Z$ extends over
$z$.
\end{proof}

\subsection{Turning points}

We propose a notion of \emph{turning points}, analogous to the corresponding
objects in the holomorphic setting.

\begin{hypothesis} \label{H:divisorial4}
Let $\Xbar$ be a smooth irreducible projective surface over $k$,
and let $K_{\overline{X}}$ denote a canonical divisor on $\Xbar$.
Let $D$ be a strict normal crossings divisor on $\Xbar$,
and put $X = \Xbar \setminus D$.
Let $\calE$ be an overconvergent isocrystal of rank $d$ on $X$.
\end{hypothesis}

\begin{defn} \label{D:hidden turning point}
Let $z$ be a nonsmooth point of $D$, and let 
$Z_1, Z_2$ be the components of $D$
containing $z$.
Let $t_1, t_2$ be local parameters for $Z_1, Z_2$ at $z$.
Define the functions $B_1(\calE,r), \dots, B_d(\calE,r)$ as in
Theorem~\ref{T:convex variation}; for $s \in [0,1]$, put
$f_i(s) = B_i(\calE,(1-s,s))$.
We say that $z$ is a \emph{hidden turning point} if $f_i(s)$
is \emph{not} affine in $s$ for some $i \in \{1,\dots,d\}$.
\end{defn}

\begin{prop} \label{P:hidden turning point}
In Definition~\ref{D:hidden turning point}, 
let $f_i'(0)$ denote the right slope of $f_i$ at $s=0$. Then
$f_i'(0) \leq f_i(1) - f_i(0)$,
with equality for all $i$
if and only if $z$ fails to be a hidden turning
point.
\end{prop}
\begin{proof}
This is evident from the fact that $f_i$ is convex
(Theorem~\ref{T:convex variation}).
\end{proof}

\begin{defn}
Let $z$ be a smooth point of $D$, and let $Z$ be the component of $D$
containing $z$. By Theorem~\ref{T:monoton2},
for each $i \in \{1,\dots,d\}$ such
that either $i=d$ or $b_i(\calE,Z) > b_{i+1}(\calE,Z)$,
we have $b'_1(\calE,z) + \cdots + b'_i(\calE,z) + \ell_i(\calE,Z) \leq 0$.
We say $z$ is an 
\emph{exposed turning point} if this inequality is strict for at least one
$i$.
\end{defn}

It is natural to mention a variant of 
Theorem~\ref{T:div subhar} phrased in terms of intersection theory
rather than valuations.

\begin{defn}
For each component $Z$ of $D$, let $\Swan(\calE,Z)$ denote the differential
Swan conductor of $\calE$ along $Z$,
and define $\ell(\calE, Z)$ as in Theorem~\ref{T:div subhar}.
Define the \emph{Swan divisor} of $\calE$ on $\Xbar$ as the divisor
\[
\Swan(\calE) = \sum_{Z \in D} \Swan(\calE, Z) Z.
\] 
\end{defn}

\begin{lemma} \label{L:convert to intersections}
Under Hypothesis~\ref{H:divisorial4},
for each component $Z$ of $D$,
\[
Z \cdot (\Swan(\calE) + \ell(\calE, Z)(K_{\overline{X}} + D)) \geq 
(2g(Z)-2)\ell(\calE, Z)
+ Z^2 \Swan(\calE,Z) +
\sum_{z \in Z} (\Swan_Z'(\calE,z) + \ell(\calE, Z)).
\]
Moreover, equality holds
if $\calE$ has no turning points on $Z$.
\end{lemma}
\begin{proof}
Rewrite the left side as
\[
Z^2 \Swan(\calE,Z) + \ell(\calE,Z)(Z \cdot K_{\overline{X}} + Z^2)
+ \sum_{Z'} (\Swan(\calE,Z') + \ell(\calE,Z))(Z \cdot Z'),
\]
where $Z'$ runs over the components of $D$ other than $Z$.
By adjunction, $Z \cdot K_{\overline{X}} + Z^2 = 2g(Z) - 2$.

Since we assumed $D$ is a strict normal crossings divisor,
$Z \cap Z'$ never contains more than one point.
For each $z \in Z$ occurring as $Z \cap Z'$ for some $Z'$,
by Proposition~\ref{P:hidden turning point}, 
we have $\Swan(\calE,Z') \geq \Swan_Z'(\calE,z)$
with equality if $z$ fails to be a hidden turning point.
More explicitly, if we identify $Z,Z'$ with the divisors
$Z_1, Z_2$ of Proposition~\ref{P:hidden turning point},
then $f_d(s) = (1-s) \Swan(\calE,x,z, s/(1-s))$, so
$\Swan'_Z(\calE,z) = f_d'(0) + \Swan(\calE,Z) \leq f_d(1) - f_d(0)
+ \Swan(\calE,Z) = \Swan(\calE,Z')$.

For each $z \in Z$ not occurring as $Z \cap Z'$ for any $Z'$,
we have by Theorem~\ref{T:monoton2}
that $\Swan_Z'(\calE,z) + \ell(\calE,Z) \leq 0$, with equality
if $z$ fails to be an exposed turning point.
This yields the claimed results.
\end{proof}

\begin{theorem} \label{T:div subhar2}
Under Hypothesis~\ref{H:divisorial4},
for each component $Z$ of $D$,
\begin{equation} \label{eq:div subhar2}
Z \cdot (\Swan(\calE) + \ell(\calE, Z)(K_{\overline{X}} + D)) \geq 0.
\end{equation}
\end{theorem}
\begin{proof} 
This holds by combining Lemma~\ref{L:convert to intersections}
with Theorem~\ref{T:div subhar}.
\end{proof}

\begin{defn}
We say $\calE$ is \emph{clean} on $\Xbar$ if it has no turning points,
either hidden or exposed.
\end{defn}

\begin{remark}
It is not immediately obvious that one can always blow up $\Xbar$
in order to make $\calE$ clean. One would like to argue that the total
multiplicity of the turning points never increases and can be forced to
decrease by a certain series of point blowups. However, one may be forced
to temporarily increase the total multiplicity by blowing up an exposed
turning point along a divisor $Z$ with $\ell(\calE,Z) > 0$.
(This does not happen with Kato's notion of cleanness from \cite{kato},
which is preserved by blowing up. Liang Xiao points out that in fact
our definition of cleanness, when restricted to sheaves of rank 1,
is more restrictive than Kato's definition.)
\end{remark}

\begin{question}
If $\calE$ is clean, can one assert an Euler characteristic formula
for $\calE$ analogous to the Grothendieck-Ogg-Shafarevich formula
for curves? (For the $p$-adic version for curves, see for instance
\cite[Theorem~4.3.1]{kedlaya-weilii}.)
Such a formula would involve not only contributions from the components
of $D$, but also from the pairwise intersections of components.
\end{question}

\begin{question}
If $\calE = f_* \calO_Y$ for $f: Y \to X$ a finite \'etale morphism,
and $\calE$ is clean, can one form a finite cover $\overline{f}: \Ybar \to
\Xbar$ extending $f$ such that $\Ybar$ has only
mild singularities? For instance, if $f$ is Galois
and abelian, it should be possible to ensure that $\Ybar$ has only 
quotient singularities; something along these lines has been established
by Kato \cite{kato}, although some work may be needed to compare our 
construction with his.
\end{question}

\section{Results for lisse $\ell$-adic sheaves}
\label{sec:l-adic}

In this section, we describe how to define differential ramification
breaks and Swan conductors
for lisse $\ell$-adic \'etale sheaves, and how some of the variational results
in the $p$-adic case may be carried over.
Throughout this section, retain
Hypotheses~\ref{H:etale1} and~\ref{H:etale2}.

\setcounter{theorem}{0}
\begin{hypothesis}
Throughout this section,
let $\ell$ be a prime different from $p$, and let $E$ be a finite 
extension of $\QQ_\ell$. 
\end{hypothesis}

\subsection{Defining the ramification breaks}

\begin{defn}
Let $v$ be a divisorial valuation $v$ on $k(X)$ over $k$,
and let $I_v$ be an inertia subgroup of $v$.
The \emph{wild inertia subgroup} $W_v$ of $I_v$ is
the absolute Galois group of the maximal tamely
ramified extension of $k(X)_v$. The group $W_v$ is a pro-$p$-group,
whereas the quotient $I_v/W_v$ is congruent to
$\prod_{\ell \neq p} \ZZ_\ell$.
\end{defn}

\begin{defn} \label{D:ell-adic}
Let $\rho: \pi_1(X, \overline{x}) \to \GL(V)$
be a continuous homomorphism for $V = V(\rho)$ a finite-dimensional
$E$-vector space, corresponding to a lisse $E$-sheaf $\calE$ on $X$.
Filter the inertia group $I_v$ as in \cite[Definition~3.5.12]{kedlaya-swan1}.
For $\rho$ irreducible, define the \emph{differential highest break}
$b(\rho, v)$ of $\rho$ along $v$ to be the maximal $r$ such that
$I_v^r \not\subset \ker(\rho)$.
For general $\rho$, let $\rho_1,\dots,\rho_n$ be the irreducible
constituents of $\rho$, and define the 
\emph{differential ramification breaks}
$b_1(\rho,v) \geq \dots \geq b_d(\rho,v)$ 
(or $b_1(\calE,v) \geq \dots \geq b_d(\calE,v)$)
of $\rho$ to be the elements
of the multiset consisting of $b(\rho_i,v)$ with multiplicity $\dim(\rho_i)$.
Define the \emph{differential Swan conductor}
$\Swan(\rho, v)$ 
(or $\Swan(\calE, v)$)
of $\rho$ along $v$ to be the sum 
$\sum_{i=1}^d b_i(\rho,v)$.

As in Definition~\ref{D:normalization}, the previous definition
gives the differential ramification breaks and differential Swan 
conductor in their \emph{natural normalization}.
If desired, we may instead normalize with respect
to any $t \in k(X)$ for which $v(t) \neq 0$.
\end{defn}

Unlike in the $p$-adic case, the differential ramification breaks of
an $\ell$-adic representation of $\pi_1(X, \overline{x})$ are not obtained
by first constructing a corresponding isocrystal. Consequently, it is not
immediate that variational
properties of differential ramification breaks of representations
can be transferred to the $\ell$-adic case. The remainder of this
section is devoted to making such transfers; we start with a few 
useful remarks.

\begin{remark} \label{R:pass to reduction}
With notation as in Definition~\ref{D:ell-adic},
choose a $\rho$-stable $\gotho_E$-lattice $T$ of $V$,
and let $\overline{\rho}: \pi_1(X, \overline{x}) \to \GL(T/\gothm_E T)$
be the resulting residual representation. 
Then the image in $\GL(T)$ of the
pro-$p$-group $W_v$ has trivial intersection with the pro-$\ell$-group
$\ker(\GL(T) \to \GL(T/\gothm_E T))$, and so injects 
into $\GL(T/\gothm_E T)$. 
Consequently, if we use the same procedure as in Definition~\ref{D:ell-adic}
to define the differential ramification breaks and 
Swan conductor of a mod $\ell$ representation
of $\pi_1(X, \overline{x})$, then these quantities are the same
for a $\gotho_E$-representation as for its mod $\ell$ reduction.
\end{remark}

\begin{remark} \label{R:etale lift}
In Remark~\ref{R:pass to reduction}, if the representation
$\overline{\rho}$ lifts to a \emph{discrete} representation
$\pi_1(X, \overline{x}) \to \GL(T)$
(i.e., a representation with open kernel), then we can
generate an overconvergent $F$-isocrystal which computes the 
differential ramification breaks of $\overline{\rho}$,
using Theorem~\ref{T:crew functor2}.
\end{remark}

\subsection{Integral polyhedrality}

In this section, we establish an analogue of Theorem~\ref{T:convex variation}
for $\ell$-adic sheaves.
\begin{theorem} \label{T:convex variation2}
Under Hypothesis~\ref{H:variation}, let 
$\calE$ be a lisse \'etale $E$-sheaf on $X$.
For $i=1, \dots, d$ and $r \in T \cap \QQ^n$, let
$b_i(\calE,r)$ denote the $i$-th largest differential ramification break
of $\calE$ along $v_r$, normalized with respect to $t_1 \cdots t_n$.
Put $B_i(\calE,r) = b_1(\calE,r) + \cdots + b_i(\calE,r)$.
Then the functions $d! B_i(\calE,r)$ and $B_d(\calE,r)$ are
continuous, convex, and integral polyhedral on $T$.
\end{theorem}
\begin{proof}
By Remark~\ref{R:pass to reduction}, we may replace $\calE$ by a
locally constant \'etale $\FF$-sheaf, where $\FF$ is the residue field
of $E$, and prove the same result.
Let $G$ be the image of $\pi_1(X, \overline{x})$ in $\GL_d(\FF)$, and let
$H$ be a $p$-Sylow subgroup of $G$. Let $f:Y \to X$ be a finite \'etale cover 
such that for some geometric point $\overline{y}$ of $Y$ over $\overline{x}$,
$\pi_1(Y, \overline{y}) = \overline{\rho}^{-1}(H)$.
Put $\calF = f_* f^* \calE$, which corresponds to the representation
$\tau = \Ind^G_H \Res^G_H \overline{\rho}$. 
Put $m = [G:H]$. Then for each divisor $Z$ on $X$,
\[
b_{(m-1)i + 1}(\calF, Z) = \cdots = b_{mi}(\calF,Z) = b_i(\calE,Z)
\]
since the differential ramification breaks only depend on the action of $H$.
On the other hand, $\Res^G_H \overline{\rho}$
is a mod $\ell$ representation of the group $H$ whose order is prime to $\ell$.
It is thus liftable to $\gotho_E$, as then is its induction $\tau$.
We may thus apply Remark~\ref{R:etale lift} to deduce from
Theorem~\ref{T:convex variation} that
$m d! B_i(\calE,r) = d! B_{mi}(\calF,r)$ and $m B_d(\calE,r) = B_{md}(\calF,r)$
are continuous, convex, and integral polyhedral.

To conclude, note that on one hand, $d! B_i(\calE,r)$ and $B_d(\calE,r)$
are continuous, convex, and polyhedral by the previous paragraph. On the
other hand, for each $r \in T \cap \QQ^n$, $d! B_i(\calE,r)$ and $B_d(\calE,r)$
take values in $\ZZ + \ZZ r_1 + \cdots + \ZZ r_n$ by the Hasse-Arf
property of differential Swan conductors \cite[Theorem~2.8.2]{kedlaya-swan1}.
Hence by Theorem~\ref{T:integral polyhedral} (or an elementary argument),
$d! B_i(\calE,r)$ and $B_d(\calE,r)$ are integral polyhedral.
\end{proof}

\begin{remark}
Although the above argument suffices for our purposes, it is worth mentioning
another lifting construction that may occasionally be useful.
Let $\calE$ be a locally constant \'etale $\FF$-sheaf on  $X$,
where $\FF$ is the residue field of $E$.
Let $G$ be the image of $\pi_1(X, \overline{x})$ in $\GL_d(\FF)$.
For $S$ any ring, let $R_S(G)$ denote the Grothendieck ring of finite 
$S[G]$-modules. Then the canonical map $R_{\gotho_E}(G) \to R_\FF(G)$
is surjective by \cite[Chapter~16, Theorem~33]{serre-rep}, so the given
$\FF$-representation of $G$ lifts to a virtual $\gotho_E$-representation
of $G$. We may then convert each factor of the virtual representation into
an overconvergent $F$-isocrystal as in Remark~\ref{R:etale lift}.
Unfortunately, since this representation is only virtual, one cannot use this
argument to deduce convexity or polyhedrality.
\end{remark}

\subsection{Subharmonicity and monotonicity}

We may also obtain the subharmonicity and monotonicity results
for surfaces, by using  the same technique as in Theorem~\ref{T:convex variation2}
to reduce to Theorems~\ref{T:div subhar} and~\ref{T:monoton2}, respectively.
(Initially one only proves $\ell(\calE,Z) 
\in \QQ \cap [0,d]$ because of the division
by $m$ in the argument of Theorem~\ref{T:convex variation2}, but the integral
polyhedrality of Theorem~\ref{T:convex variation2} forces $\ell(\calE,Z)
\in \ZZ$,
so there is no problem.)
\begin{theorem} \label{T:div subhar ell}
Assume that $k = k_0$ is algebraically closed.
Let $\Xbar$ be a smooth irreducible projective surface over $k$,
let $Z$ be a smooth divisor on $\Xbar$,
and let $v_0$ be the divisorial valuation on $k(\Xbar)$ measuring
order of vanishing along $Z$.
Let $X$ be an open dense subscheme of $\Xbar$, and let
$\calE$ be a lisse \'etale $E$-sheaf on $X$.
Define $b_j(\calE,z,x,r)$ for $j=1,\dots,d$
and $\Swan_Z(\calE,z,x,r)$ as in Definition~\ref{D:local variation3}.
\begin{enumerate}
\item[(a)]
For each $z \in Z$, the functions
$b_j(\calE,z,x,r)$ for $j=1,\dots,d$
and $\Swan_Z(\calE,z,x,r)$ are affine in a neighborhood
of $r=0$.
\item[(b)]
Let $\Swan_Z'(\calE,z)$ be the right slope of $\Swan_Z(\calE,z,x,r)$ at $r=0$.
Then there exists $\ell(\calE,Z) \in \{0, 1,\dots,d\}$ such that
$\Swan_Z'(\calE,z) = -\ell(\calE,Z)$ 
for all but finitely many $z \in Z$.
\item[(c)]
With notation as in (b), we have
\[
\sum_{z \in Z} (\Swan_Z'(\calE,z) + \ell(\calE,Z)) \geq (2-2g(Z))\ell(\calE,Z)
- Z^2 \Swan(\calE,Z),
\]
where $g(Z)$ denotes the genus of $Z$, and $Z^2$ denotes the
self-intersection of $Z$ on $\Xbar$.
\item[(d)]
If $z$ is a smooth point of $Z \cup (\Xbar \setminus X)$, 
then $\Swan'_Z(\calE,z) + \ell(\calE,Z) \leq 0$.
\end{enumerate}
\end{theorem}

\begin{theorem} \label{T:div subhar2 ell}
With hypotheses as in Theorem~\ref{T:div subhar ell},
for each component $Z$ of $D$,
\[
Z \cdot (\Swan(\calE) + \ell(\calE, Z)(K_{\overline{X}} + D)) \geq 0.
\]
\end{theorem}
\begin{proof} 
This follows from Theorem~\ref{T:div subhar ell} by the same
argument as in Theorem~\ref{T:div subhar2}.
\end{proof}

\begin{remark}
It should be possible to use Theorem~\ref{T:div subhar ell}
to give an independent
derivation of the semicontinuity theorem in \'etale cohomology
\cite{laumon}. We leave this as an exercise for the interested reader.
\end{remark}


\begin{thebibliography}{99}

\bibitem{abbes-saito1}
A. Abbes and T. Saito, Ramification of local fields with
imperfect residue fields, \textit{Amer. J. Math.} \textbf{124}
(2002), 879--920.

\bibitem{abbes-saito2}
A. Abbes and T. Saito, Ramification of local fields with
imperfect residue fields II, \textit{Doc. Math.} Extra Volume
(2003), 5--72.

\bibitem{andre}
Y. Andr\'e, Structure des connexions m\'eromorphes formelles
de plusieurs variables et semi-continuit\'e de
l'irregularit\'e, \textit{Invent. Math.} \textbf{170} (2007), 
147--198.

\bibitem{arabia}
A. Arabia, Rel\`evements des alg\`ebres lisses et de leurs morphismes,
\textit{Comment. Math. Helv.} \textbf{76} (2001), 607--639.

\bibitem{berthelot}
P. Berthelot, Cohomologie rigide et cohomologie rigide \`a support
    propre. Premi\`ere partie,
Pr\'epublication IRMAR 96-03, available at
    \texttt{http://perso.univ-rennes1.fr/pierre.berthelot/}.

\bibitem{crew}
R. Crew, $F$-isocrystals and $p$-adic representations,
in \textit{Algebraic geometry (Brunswick, Maine, 1985), part 2},
Proc. Sympos. Pure. Math. 46, Amer. Math. Soc., Providence, 1987,
111--138.

\bibitem{dejong}
A.J. de Jong, Smoothness, semi-stability and alterations,
\textit{Publ. Math. IH\'ES} \textbf{83} (1996), 51--93.

\bibitem{grosse-klonne}
E. Grosse-Kl\"onne, Rigid analytic spaces with overconvergent
structure sheaf, \textit{J. reine angew. Math.} \textbf{519} (2000), 73--95.

\bibitem{sga1}
A. Grothendieck et al., \textit{Rev\^etements \'etales et groupe fondamental
(SGA 1)}, Lecture Notes in Math. 224, Springer-Verlag, Berlin, 1971.

\bibitem{kato}
K. Kato,
Class field theory, $\mathcal{D}$-modules, and ramification on
higher dimensional schemes, part I, 
\textit{Amer. J. Math.} \textbf{116} (1994), 757--784.

\bibitem{kedlaya-mono-over}
K.S. Kedlaya,
Local monodromy for $p$-adic differential equations: an overview,
\textit{Intl. J. of Number Theory} \textbf{1} (2005), 109--154.

\bibitem{kedlaya-weilii}
K.S. Kedlaya, Fourier transforms and $p$-adic ``Weil II'', 
\textit{Compos. Math.} \textbf{142} (2006), 1426--1450. 

\bibitem{kedlaya-part1}
K.S. Kedlaya, Semistable reduction for overconvergent $F$-isocrystals, I:
    Unipotence and logarithmic extensions,
\textit{Compos. Math.} \textbf{143} (2007), 1164--1212.

\bibitem{kedlaya-swan1}
K.S. Kedlaya, Swan conductors for $p$-adic differential modules, I:
A local construction, 
\textit{Alg. and Num. Theory} \textbf{1} (2007), 269--300.

\bibitem{kedlaya-part2}
K.S. Kedlaya, Semistable reduction for overconvergent $F$-isocrystals, II:
A valuation-theoretic approach, 
\textit{Compos. Math.} \textbf{144} (2008), 657--672. 

\bibitem{kedlaya-part3}
K.S. Kedlaya, Semistable reduction for overconvergent $F$-isocrystals, III:
Local semistable reduction at monomial valuations, arXiv preprint 
\texttt{math/0609645v3} (2008); to appear in \textit{Compos. Math.}

\bibitem{kedlaya-part4}
K.S. Kedlaya, Semistable reduction for overconvergent $F$-isocrystals, IV:
Local semistable reduction at nonmonomial valuations, arXiv preprint 
\texttt{0712.3400v2} (2008).

\bibitem{kedlaya-course}
K.S. Kedlaya, $p$-adic differential equations (version of 31 Oct 08),
course notes available at \texttt{http://math.mit.edu/\~{}kedlaya/papers/}.

\bibitem{kedlaya-goodformal1}
K.S. Kedlaya,
Good formal structures for flat meromorphic connections, I: Surfaces, 
arXiv preprint \texttt{0811.0190v1} (2008). 

\bibitem{kedlaya-xiao}
K.S. Kedlaya and L. Xiao, Differential modules on $p$-adic polyannuli,
arXiv preprint \texttt{0804.1495v3} (2008);
to appear in \textit{J. Inst. Math. Jussieu}.

\bibitem{kedlaya-tynan}
K.S. Kedlaya and P. Tynan, Detecting integral polyhedral functions,
arXiv preprint \texttt{0811.3241v1} (2008).

\bibitem{laumon}
G. Laumon, Semi-continuit\'e du conducteur de Swan (d'apr\`es P. Deligne),
in \textit{The Euler-Poincar\'e characteristic}, Ast\'erisque 83, Soc. Math.
France, 1981, 173--219.

\bibitem{lazard}
M. Lazard, Les z\'eros des fonctions analytiques d'une variable
sur un corps valu\'e complet, \textit{Publ. Math. IH\'ES} \textbf{14}
(1962), 47--75.

\bibitem{mebkhout}
Z. Mebkhout, Sur le Th\'eor\`eme de semi-continuit\'e des \'equations
diff\'erentielles, \textit{Ast\'erisque} \textbf{130} (1985), 365--417.

\bibitem{meredith}
D. Meredith, Weak formal schemes, \textit{Nagoya Math. J.}
\textbf{45} (1971), 1--38.

\bibitem{sabbah}
C. Sabbah,
\'Equations diff\'erentielles \`a points singuliers irr\'eguliers 
et ph\'enom\`ene de Stokes en dimension 2,
\textit{Ast\'erisque} \textbf{263} (2000).

\bibitem{serre-rep}
J.-P. Serre, \textit{Linear Representations of Finite Groups},
Graduate Texts in Math. 42, Springer-Verlag, New York, 1977.

\bibitem{thuillier}
A. Thuillier, Th\'eorie du potentiel sur les courbes en g\'eom\'etrie
analytique non archim\'edienne. Applications \`a la theorie d'Arakelov,
thesis, Universit\'e de Rennes 1, 2005.

\bibitem{tsuzuki-finite}
N. Tsuzuki, Finite local monodromy of overconvergent unit-root
$F$-isocrystals on a curve, \textit{Amer. J. Math.} \textbf{120} (1998),
1165--1190.

\bibitem{tsuzuki-duke}
N. Tsuzuki, Morphisms of $F$-isocrystals and the finite monodromy
theorem for unit-root $F$-isocrystals, \textit{Duke Math. J.} \textbf{111}
(2002), 385--418.

\bibitem{xiao}
L. Xiao, On ramification filtrations and $p$-adic differential equations, I:
equal characteristic case,
arXiv preprint \texttt{arXiv:0801.4962v2} (2008).

\bibitem{xiao2}
L. Xiao, On ramification filtrations and $p$-adic differential equations, II:
mixed characteristic case,
preprint available at \texttt{http://math.mit.edu/\~{}lxiao/}.


\end{thebibliography}
\end{document}